\documentclass[3p,longnamesfirst,sort&compress]{elsarticle}
\usepackage{amsmath,amsthm,amssymb,mathrsfs}
\usepackage{thmtools,thm-restate}
\usepackage{enumerate}
\usepackage{slashed}
\usepackage{txfonts,dsfont}
\usepackage[breaklinks,colorlinks]{hyperref}
\usepackage{nameref}
\usepackage{cleveref}

\makeatletter
\newcommand{\autorefcheckize}[1]{%
  \expandafter\let\csname @@\string#1\endcsname#1%
  \expandafter\DeclareRobustCommand\csname relax\string#1\endcsname[1]{%
    \csname @@\string#1\endcsname{##1}\wrtusdrf{##1}}%
  \expandafter\let\expandafter#1\csname relax\string#1\endcsname
}
\makeatother
\declaretheorem[numberwithin=section]{theorem}
\declaretheorem[sibling=theorem, name=Lemma]{lem}
\declaretheorem[sibling=theorem, name=Corollary]{cor}

\declaretheorem[sibling=theorem, name=Remark]{rem}
\declaretheorem[numberwithin=section, name=Definition]{defn}

\numberwithin{equation}{section}

\newcommand{\norm}[1]{\left\lVert#1\right\rVert}
\newcommand{\abs}[1]{\left\lvert#1\right\rvert}
\newcommand{\set}[1]{\left\{#1\right\}}
\newcommand{\hin}[2]{\left\langle#1,#2\right\rangle}

\newcommand{\tr}{\mathrm{tr}}

\newcommand{\mV}{\mathcal{X}}
\newcommand*{\rmn}[1]{\romannumeral#1}
\newcommand*{\Rmn}[1]{\uppercase\expandafter{\romannumeral#1}}

\DeclareMathOperator{\Div}{div}

\journal{***}

\allowdisplaybreaks

\begin{document}

\begin{frontmatter}
\title{Optimal Liouville theorems for the Lane-Emden equation on  Riemannian manifolds}

\author[buct]{Jie He}
\ead{hejie@amss.ac.cn}

\author[xtu]{Linlin Sun}
\ead{sunll@xtu.edu.cn}

\author[gzhu,amss,ucas]{Youde Wang\corref{wyd}}
\ead{wyd@math.ac.cn}

\address[buct]{School of Mathematics and Physics, Beijing University of Chemical Technology, Beijing 100029, China}

\address[xtu]{School of Mathematics and Computational Science, Xiangtan University, Xiangtan 411105, China}

\address[gzhu]{School of Mathematics and Information Sciences, Guangzhou University, Guangzhou 510006, China}
\address[amss]{State Key Laboratory of Mathematical Sciences (SKLMS), Academy of Mathematics and Systems Science, \\Chinese Academy of Sciences, Beijing 100190, China}
\address[ucas]{School of Mathematical Sciences, University of Chinese Academy of Sciences, Beijing 100049, China}

\cortext[wyd]{Corresponding author.}

\begin{abstract}
This paper study the Liouville-type theorems and universal estimates of positive solutions for the $p$-Laplacian Lane-Emden type equations on Riemannian manifolds. In particular, we prove an optimal Liouville theorem, demonstrating the nonexistence of positive solutions for subcritical $p$-Laplacian Lane-Emden equations on complete noncompact Riemannian manifolds with nonnegative Ricci curvature. This resolves a longstanding open problem in the field. Our approach features two key innovations:
(i) an adapted Nash-Moser iteration scheme yielding universal log-gradient estimates for the degenerate elliptic operator, and
(ii) new integral estimates that produce sharp nonexistence results.
These techniques enable us to extend the  Euclidean quasilinear theory of Serrin-Zou (Acta Math. 189, 79-142, 2002) to general Riemannian manifolds.
\end{abstract}


\end{frontmatter}

\section{Introduction}

In this paper we are concerned with the following equation
\begin{align*}
\Delta_p u + f(u)=0,\quad p>1
\end{align*}
which is defined on a complete Riemannian manifold  of a Ricci curvature lower bound. In the case $p=2$ and $f(u)\equiv u^{\alpha}$, this equation is just the well-known Lane-Emden equation.

In the seminal paper \cite{GidSpr81global} Gidas and Spruck proved that there are no positive solutions to the Lane-Emden equation (see \cite[Theorem 1.1]{GidSpr81global}), i.e.,
\begin{align}\label{eq:lane-emden-2}
-\Delta u=u^{\alpha},\quad\mbox{in} ~~\mathbb{R}^n
\end{align}
if $n>2$ and
\begin{align*}
1\leq\alpha<\dfrac{n+2}{n-2}.
\end{align*}
This result can be regarded as an extension of the classical Liouville theorem for harmonic functions defined on the Euclidean space states that \emph{any real, positive, and entire harmonic function must be a constant}.

Gidas and Spruck's result can be generalized to quasilinear elliptic equations. For example, Serrin and Zou \cite[Corollary \Rmn{2}]{SerZou02cauchy} proved that the classical Lane-Emden-Fowler equation
\begin{align}\label{eq:lane-emden}
-\Delta_pu=u^{\alpha}
\end{align}
admits no positive solution in $\mathbb{R}^n$ if
\begin{align*}
  0<\alpha <\dfrac{(n+1)p-n}{n-p}\quad \text{and} \quad 1<p<n.
\end{align*}
Actually, the above Liouville theorems established in \cite{GidSpr81global, SerZou02cauchy} for \eqref{eq:lane-emden-2} and \eqref{eq:lane-emden} have been extended to the cases
\begin{align*}
    -\infty <\alpha <\dfrac{n+2}{n-2}\quad\quad \text{and} \quad\quad -\infty <\alpha <\dfrac{(n+1)p-n}{n-p}
\end{align*}
respectively, for details we refer to \cite{WanWei23on, HeWanWei24gradient}. It should be mentioned that this result fails for any
$$\alpha\geq\dfrac{(n+1)p-n}{n-p}.$$
The typical example is the Emden solution on $\mathbb R^n$ when $\alpha=\frac{(n+1)p-n}{n-p}$
\begin{align*}
u_{\lambda, x_0}(x)\coloneqq\left(\dfrac{\lambda^{\frac{1}{p-1}}n^{\frac{1}{p}}\left(\frac{n-p}{p-1}\right)^{\frac{p-1}{p}}}{\lambda^{\frac{p}{p-1}}+\abs{x-x_0}^\frac{p}{p-1}}\right)^{\frac{n-p}{p}},
\end{align*}
where $\lambda>0$ is any positive parameter and $x_0\in\mathbb R^n$.

The equations \eqref{eq:lane-emden-2} and \eqref{eq:lane-emden} on an Euclidean space $\mathbb{R}^n$ have been extensively studied by many mathematicians. Without being exhaustive with the huge amount of references concerned with this topic, let us mention the works \cite{BidMar89local, BidMarVer91nonlinear, CafGidSpr898asymptotic, CheLi91classification, Frank24, NiSer86nonexistence, LinMa24, MaOu23liouville, Phuc08} and references therein. Besides, in \cite{GidSpr81global, SerZou02cauchy} the following equation
\begin{align}\label{general}
\Delta u + f(u)=0, \quad \mbox{in}~~\mathbb{R}^n
\end{align}
was also be studied and some universal estimates for positive solutions to \eqref{general} were established (also see \cite{Dancer, PQS, QS}). In particular, Yanyan Li and Lei Zhang \cite{Li-Z} in 2003 studied the equation and proved the following results \cite[Theorem 1.3]{Li-Z}: Suppose that $f$ is locally bounded on $(0, +\infty)$ and $s^{-\frac{n+2}{n-2}}f(s)$ is non-increasing on $(0, +\infty)$. Assume that $u$ is a positive classical solution of the above equation \eqref{general} in $\mathbb{R}^n$ with $n \geq 3$, then either for some $b>0$,
$$bu(x)= \left(\dfrac{\mu}{1 + \mu^2\abs{x - \bar{x}}^2}\right)^{\frac{n-2}{2}}$$
or $u \equiv a$ for some $a > 0$ such that $f(a)=0$.
\medskip

The Lane-Emden type equation also appears in fluid mechanics and conformal differential geometry. When domain manifold $M^n$ is the two sphere $\mathbb{S}^2$, the equation \eqref{eq:lane-emden-2} on $\mathbb{S}^2$ has a deep relationship with the stationary solutions to Euler's equation on $\mathbb{S}^2$ (such as zonal flows, Rossby-Haurwitz planetary waves and Stuart-type vortices). We refer to \cite{CCKW, CG, CK} and references therein for recent developments of stationary solutions of Euler's equation on $\mathbb{S}^2$.

On the other hand, the renowned Yamabe equation, a pivotal problem in conformal geometry, seeks to ascertain the existence of a conformal metric with constant scalar curvature on any closed Riemannian manifold. Specifically, for $n\geq3$, the question arises: can we identify a conformal metric $\tilde g=u^{\frac{4}{n-2}}g$ such that the scalar curvature of $\tilde g$ remains constant? Consequently, the Yamabe problem is tantamount to solving the following partial differential equation
\begin{align*}
-\Delta_g u+\dfrac{n-2}{4(n-1)}Ru=\lambda u^{\frac{n+2}{n-2}}
\end{align*}
where $R=R(g)$ denotes the scalar curvature associated with the initial metric $g$ and $\lambda$ represents a constant.
\vspace{2ex}

In fact, \emph{whether or not Liouville theorems pertaining to the Lane-Emden-Fowler equation on Riemannian manifolds with nonnegative Ricci curvature are true for
$$-\infty <\alpha <\frac{(n+1)p-n}{n-p}\quad \mbox{and} \quad 1<p<n$$
is a long-standing open problem.}

It should be mentioned that Gidas and Spruck \cite[Theorem 1.2]{GidSpr81global} demonstrated that no positive solution to the Lane-Emden equation \eqref{eq:lane-emden-2} exists on any complete Riemannian manifold of dimension $n$ with nonnegative Ricci curvature, under the condition
\begin{align*}
1\leq\alpha<\dfrac{n+2}{n-2}.
\end{align*}
It is noteworthy that their proof employed a crucial estimate for the Laplacian cutoff function $\zeta$ which satisfies
\begin{align*}
\zeta\in C_0^{\infty}\left(B_{R}\right),\quad 0\leq \zeta\leq 1,\quad \zeta\vert_{B_{R/2}}=1,\quad \abs{\nabla\zeta}\leq\dfrac{C}{R},\quad\abs{\Delta\zeta}\leq\dfrac{C}{R^2}.
\end{align*}
It is highly nontrivial for us to prove the existence of such kind of cutoff functions on complete Riemannian manifolds. Indeed, if the Ricci curvature is nonneagtive, then one can prove the existence of Laplacian cutoff functions from Riemannian rigidity theory (cf. \cite[Theorem 6.33]{CheCol96lower} or \cite[Lemma 1.5]{WanZhu19structure}), with a careful scaling argument (see \cite[Theorem 2.2, (b)]{Gun16sequences} or \cite[Corollary 2.3]{BiaSet18laplacian}). Recently, Catino and Monticelli \cite{Catino} gave a new proof of the Liouville theorem \cite[Theorem 1.2]{GidSpr81global} on the equation \eqref{eq:lane-emden-2}. In fact, Catino and Monticelli \cite[Theorem 1.7]{Catino} made also use of the same
cutoff function as in \cite[Theorem 1.2]{GidSpr81global} to prove the conclusion. Later on Lu \cite{Lu23semilinear} used maximum priciple to establish Cheng-Yau type gradient estimate and hence to prove the Liouville theorem due to Gidas and Spruck. Very recently, Ciraolo, Farina, and Polvara employed a different method to provide a very simple proof in \cite{CirFarPol24classification}.

On the other hand, very recently the first and third named author and Wei \cite{HeWanWei24gradient} have shown that \eqref{eq:lane-emden} does not admits any positive solution if $p>1$ and
\begin{align*}
-\infty<\alpha<\dfrac{(n+3)(p-1)}{n-1},
\end{align*}
but, obviously
$$\dfrac{(n+3)(p-1)}{n-1}<\dfrac{(n+1)p-n}{(n-p)^+}=\begin{cases}\frac{(n+1)p-n}{n-p},&p<n,\\
+\infty,&p\geq n.
\end{cases}$$
\medskip

\emph{The goal of this paper is to answer the above problem, i.e., to show the optimal Liouville theorem and establish some universal $\log$-gradient estimates for Lane-Emden-Fowler equation. More concretely, we will show that the equation $\Delta_p u + u^{\alpha}=0$ does not admit any positive solution if $\alpha\in (-\infty,\, p_s)$ where $p_s$ is the Sobolev critical exponent given by
$$p_s\equiv\dfrac{(n+1)p-n}{(n-p)^+}$$
(see \autoref{liouville}); and we also establish an universal $\log$-gradient estimate for positive solutions to \eqref{general} on a complete manifold (see \autoref{thm:boundedness} and \autoref{thm:global-gradient-estimate}).}

\emph{Here, by using the word ``universal”, we mean that our bounds are not only independent of any given solution under consideration but also do not require, or assume, any boundary conditions whatsoever. For instance, the exact Cheng-Yau $\log$-gradient estimates is an universal a priori gradient estimates.}
\medskip

Next, let us recall some other closely related literature on Lane-Emden equations defined on complete Riemannian manifolds. In 1975, Cheng-Yau's renowned gradient estimate for harmonic functions \cite{CheYau75} demonstrated that every positive harmonic function on any complete Riemannian manifold with nonnegative Ricci curvature is constant. This result was extended to  $p$-harmonic function by Kostschwar and Ni \cite{KotNi09local} provided that the sectional curvature is nonnegative. Subsequently, Wang and Zhang \cite{WanZha11local} initially employed the Moser iteration technique to investigate local gradient estimates and Harnack inequalities for $p$-harmonic functions on Riemannian manifolds. They established that any positive $p$-harmonic function on any complete Riemannian manifold with nonnegative Ricci curvature is constant. Notably, Wang and Zhang's result required only the lower bound of the Ricci curvature, thereby extending Cheng-Yau's result for harmonic functions to $p$-harmonic functions and significantly enhancing Kotschwar and Ni's finding.

Using Cheng-Yau's method, Li \cite{Li91gradient} considered the more general Lane-Emden type equation \eqref{eq:lane-emden-2} and obtained several gradient estimates and Harnack inequalities on Riemannian manifolds with nonnegative Ricci curvature under the assumption $1\leq\alpha < \frac{n}{n-2}, \ n\geq 4$. Notably, he established a Liouville theorem concerning the nonexistence of positive solutions. When $n\geq 5$, Ma, Huang, and Luo \cite{MaHuaLuo18gradient} proved a gradient estimate by which one can derive the nonexistence of positive solutions for
$$-\infty<\alpha < \frac{2n^2 + 9n + 6}{2n(n+2)}.$$

The De Giorgi-Nash-Moser iteration method holds significant influence not only in the analysis of PDEs but also in geometric analysis. Zhang and Zhu \cite{ZhaZhu12Yau} introduced a variant of the Bochner formula applicable to Alexandrov spaces. By combining this adapted formula with the Nash-Moser iteration method, they derived Yau's gradient estimate for harmonic functions in these settings. More recently, Bellettini \cite{Bellettini25extensions} utilized the De Giorgi iteration to provide an alternative proof for the Schoen-Simon-Yau curvature estimates and related Bernstein-type theorems, further extending the original findings to encompass the case of six-dimensional stable minimal immersions.

Inspired by Wang and Zhang's approach, the third-named author of this paper and Wei \cite{WanWei23on} employed the Moser iteration strategy to prove a local gradient estimate of the Cheng-Yau type for positive solutions to the semilinear elliptic equation \eqref{eq:lane-emden-2} on Riemannian manifolds. Specifically, they derived a Liouville theorem for the Lane-Emden equation on complete Riemannian manifolds with dimension $n$ and nonnegative Ricci curvature, given that
\begin{align*}
\alpha<\dfrac{n+1}{n-1}+\dfrac{2}{\sqrt{n(n-2)}}.
\end{align*}
This finding was later refined by the first and third named author of this paper and Wei \cite{HeWanWei24gradient}, who expanded the range of $\alpha$ to
\begin{align*}
\alpha<\dfrac{(n+3)(p-1)}{n-1}.
\end{align*}
Very recently, the third named author, A. Zhang and H. Zhao \cite{WanZhaZha24gradient} discussed the following equation
$$\Delta u + uh(\ln u)=0,\quad\mbox{in} ~~M^n,$$
where $h$ is a smooth function, and employ the Nash-Moser iteration technique to obtain some refined gradient estimates of the solutions to the above equation, if $(M^n,\,g)$ is of Ricci curvature lower bound and $h$ fulfills some suitable assumptions. This methodology can be applied to address other problems as well. Additional related results can be found in various sources \cite{HanHeWan23, HanHeWan23gradient, HeHuWan23nash, HeMaWan24, HeWan23universal, HeWanWei24gradient, HuaGuoGuo24gradient, HuaWan22gradient, WanWan24gradient, WanWan24boundedness, WanZha24gradient-1}.
\vspace{1ex}

Numerous mathematicians have studied differential inequalities on a complete manifold $(M, g)$, such as
\begin{align}\label{eq:lane-emden-3}
-\Delta_pu\geq u^{\alpha}.
\end{align}
Grigor'yan and Sun \cite{GriSun14nonnegative} examined the uniqueness of a nonnegative solution to this inequality for $p=2$. Notably, they utilized the condition of volume growth instead of relying on nonnegative Ricci curvature. Similar results hold for general $p>1$ \cite{Sun15nonexistence}. That is, there exists a no positive solution to the inequality \eqref{eq:lane-emden-3} provided $\alpha>p-1$ and for some fixed point $o\in M^n$ and some positive number $C$
\begin{align*}
\mathrm{Vol}\left(B_{R}(o)\right)\leq C R^{\frac{p\alpha}{\alpha+1-p}}\ln^{\frac{p-1}{\alpha+1-p}} R,\quad\forall R\gg 1.
\end{align*}
Consequently, using the Bishop-Gromov volume comparison theorem, it is known that every positive solution to the inequality \eqref{eq:lane-emden-3} on any complete Riemannian manifold with dimension $n$ and nonnegative Ricci curvature is constant, provided
\begin{align*}
p-1<\alpha<\dfrac{n(p-1)}{(n-p)^+}.
\end{align*}
This provides a new proof of Bidaut-Veron and Pohozaev's result \cite[Theorems 3.3 (\rmn{3}) and 3.4 (\rmn{2})]{BidPoh01nonexistence} for the Euclidean case when $1<p<n$. It should be mentioned that the range of $\alpha$ can be extended to
$$0< \alpha <\frac{n(p-1)}{n-p}$$
by Serrin and Zou \cite[Theorem \Rmn{2} (d)]{SerZou02cauchy}. Other related results can be found in various sources \cite{Hua15liouville, LiuSunXia24quasilinear, SunXiaXu22sharp, WanZha24gradient-1, GriSunVer20superlinear, WanXia16constructive, Zha15note} and references therein.

Besides, some mathematicians studied the classification to positive solutions to critical $p$-Laplace equation on a complete Riemannian manifold, which is just equation \eqref{eq:lane-emden-2} with $\alpha=p_s$, i.e.,
$$\Delta_p u + u^{p_s}=0.$$
For more details we refer to \cite{Catino, CatMonRon23on, CirFarPol24classification, Ou22on, Vet24note} and references therein. We also should mentioned that one studied the Liouville equation which is closely related to Lane-Emden equation and Yamabe equation, we refer to \cite{EreGuiLiXu25rigidity, Fazly25} and references therein.

\section{Main results and ideas of  proof}

In this paper, we focus on the Liouville theorems for quasilinear elliptic partial differential equations  on Riemannian manifolds. The specific equation under consideration is a quasilinear PDE of the form
\begin{align}\label{eq:quasi-h}
-\Delta_pu=f(u), \quad\mbox{on}~~M^n
\end{align}
where $f\in C\left(\left[0,\infty\right)\right)\cap C^1\left(\left(0,\infty\right)\right)$ and $p\in(1,\infty)$ and $\Delta_p$ denotes the $p$-Laplacian which is defined by
\begin{align*}
\Delta_pu=\begin{cases}
\Div\left(\abs{\nabla u}^{p-2}\nabla u\right),&1<p<\infty,\\
\hin{\nabla_{\nabla u}\nabla u}{\nabla u},&p=+\infty.
\end{cases}
\end{align*}
One can check that $\Delta_2=\Delta$ is the classical Laplacian-Beltrami and
\begin{align*}
\Delta_pu=\abs{\nabla u}^{p-2}\Delta u+(p-2)\abs{\nabla u}^{p-4}\Delta_{\infty}u.
\end{align*}
This type of equation is particularly important in the study of geometric analysis and has broad applications in various fields of mathematics and physics.

\begin{defn}
A function $u\in W^{1,p}_{loc}\left(M\right)\cap L^{\infty}_{loc}\left(M\right)$ is said to be a weak solution to \eqref{eq:quasi-h} if
\begin{align*}
\int_M\abs{\nabla u}^{p-2}\hin{\nabla u}{\nabla\phi}=\int_Mf(u)\phi,\quad\forall \phi\in C_0^{\infty}\left(M\right).
\end{align*}
\end{defn}

Our attention is directed toward the subcritical case. A subsequent paper \cite{SunWan25critical} will address the critical case.

The simplest global result is
\begin{theorem}\label{liouville}
Let $M^n$ be a complete and $n$-dimensional Riemannian manifold with nonnegative Ricci curvature. If $p>1$ and $\alpha<p_s$, then there is no positive and weak solution to
\begin{align*}
-\Delta_p u=u^{\alpha}, \quad\mbox{on}~~M^n.
\end{align*}
\end{theorem}
This is a special case of \autoref{thm:main4}. Since particular attention is given to the case where $p=2$ in the study of Liouville theorems pertaining to the Lane-Emden-Fowler equation on Riemannian manifolds, this means that in the case $p=2$ and $\dim(M^n)\geq 3$ the Lane-Emden equation \eqref{eq:lane-emden-2}, i.e.,
$\Delta u + u^{\alpha}=0$, does not admit any positive solution if $M^n$ is of nonnegative Ricci curvature and $\alpha<\frac{n+2}{n-2}$. In other words, we reprove the result due to Gidas and Spruck \cite{GidSpr81global}.

Another simple global result is
\begin{theorem}
Let $M^n$ be a complete and $n$-dimensional Riemannian manifold with nonnegative Ricci curvature. If $\alpha>p-1>0$, then there is no positive and weak solution to
\begin{align*}
-\Delta_p u=-u^{\alpha}, \quad\mbox{on}~~M^n.
\end{align*}
\end{theorem}
This is a special case of \autoref{main:cor-1}.

\medskip

A function $f\in C^0([0,\infty))\cap C^1((0,\infty))$ is subcritical with exponent $\alpha$ if
\begin{align}\label{eq:conditions-f}
\alpha f(t)-tf'(t)\geq0,\quad\forall t>0.
\end{align}
In the original definition of \emph{subcritical} by Serrin and Zou \cite[(1.7)]{SerZou02cauchy},
$f$ is required to be non-negative and the constant $\alpha$ in \eqref{eq:conditions-f} is required to lie in
 $\left(0,\, \frac{np}{n-p}-1\right)$. In our theorem, we do not require this constraints.  In section 4, $f$ can be negative and $\alpha$ can also be negative.

Obviously, for $f(t)= t^\alpha$ we always have
$$\alpha f(t)-tf'(t)= 0, \quad \forall t >0.$$
It is easy to verify the following functions are all subcritical with exponent $\alpha$,
\begin{enumerate}
\item $f(t)= t^\alpha\left(\log(1+t)\right)^\beta, \quad\beta <0$;
\item $f(t)= t^\alpha+at^\beta,\quad a(\alpha-\beta)>0$;
\item $f(t)= \frac{t^\alpha}{1+t^\beta}, \quad\beta>0$.
\end{enumerate}

For the general quasilinear equation \eqref{eq:quasi-h}, we have the following
\begin{theorem}\label{thm:main1}
Let $M^n$ be a complete and $n$-dimensional Riemannian manifold with nonnegative Ricci curvature. Assume $p>1$ and $f$ is subcritical with exponent $\alpha\in\left(p-1,\,\frac{(n+3)(p-1)}{n-1}\right)$.
Then every positive and weak solution to \eqref{eq:quasi-h} is constant.
\end{theorem}

The main idea is to prove the Cheng-Yau type gradient estimate. That is, motivated by Wang and Zhang's method \cite{WanZha11local}, using the Moser iteration together with Saloff-Coste's Sobolev inequalities, we prove a gradient estimate under the assumptions mentioned in \autoref{thm:main1}. As a consequence, we obtain the Liouville theorem \autoref{thm:main1}. More specifically, we prove the following Cheng-Yau type local gradient estimate.

\begin{theorem}\label{thm:boundedness}
Let $M^n$ be a complete and $n$-dimensional Riemannian manifold with the Ricci curvature bounded from below by a nonpositive constant $-(n-1)\kappa$. Assume for some constant $\delta_0<1$,
\begin{align}\label{eq:conditions-f2}
\dfrac{p-1}{n-1}\left((n+1)f(t)+2\delta_0\abs{f(t)}\right)-tf'(t)\geq0,\quad\forall t>0.
\end{align}
If $p>1$ and $u$ is a positive and weak solution to \eqref{eq:quasi-h}, then there exists a positive constant $C_{n,p,\delta_0^+}$ depending only on $n,p$ and $\delta_0^+=\max\{0, \delta_0\}\in[0, \,1)$ such that
\begin{align*}
\norm{\nabla\ln u}_{L^{\infty}\left(B_{R/4}\right)}\leq C_{n,p,\delta_0^+}\left(\dfrac{1}{R}+\sqrt{\kappa}\right),\quad\forall R>0.
\end{align*}
\end{theorem}

Moreover, based on this local gradient estimate, we obtain the following global gradient estimate.

\begin{theorem}\label{thm:global-gradient-estimate}
Let $M^n$ be a complete and $n$-dimensional Riemannian manifold with the Ricci curvature bounded from below by a nonpositive constant $-(n-1)\kappa$. Assume $p>1$ and the condition \eqref{eq:conditions-f2} holds. If $u$ is a positive and weak solution to \eqref{eq:quasi-h}, then
\begin{align*}
\norm{\nabla\ln u}_{L^{\infty}\left(M\right)}\leq \dfrac{n-1}{p-1}\sqrt{\dfrac{\kappa}{1-\delta_0^+}}.
\end{align*}
\end{theorem}

As an immediately consequence of \autoref{thm:main1}, we have
\begin{cor}\label{main:cor-1}
Let $M^n$ be a complete and $n$-dimensional Riemannian manifold with nonnegative Ricci curvature. If $p>1$ and one of the following conditions holds
\begin{enumerate}[(a)]
\item $f$ is nonnegative and \eqref{eq:conditions-f} holds for some $\alpha<\frac{(n+3)(p-1)}{n-1}$,
or
\item $f$ is nonpositive and \eqref{eq:conditions-f} holds for some $\alpha>p-1$,
\end{enumerate}
then every positive and weak solution to \eqref{eq:quasi-h} is constant.
\end{cor}

For the remaining cases
$$\alpha\geq\frac{(n+3)(p-1)}{n-1},$$
it seems that the previous method does not work. To handle this case, we follow the approach by Gidas and Spruck \cite{GidSpr81global} for $p=2$ and Serrin and Zou \cite{SerZou02cauchy} for general $p>1$ via the method of vector fields motivated by Obata \cite{Oba71conjecture}. This trick introduced by Obata \cite{Oba71conjecture} has many applications, see \cite{LinMa24,MaWu23liouville,MaWuZha23liouville} and references therein.
\medskip

We obtain the following

\begin{theorem}\label{thm:main2}
Let $M^n$ be a complete Riemannian manifold with dimension $n$ and nonnegative Ricci curvature. Assume $p>1$ and \eqref{eq:conditions-f} holds for some $\alpha\in(0,\,p_s)$. If we assume additionally that $f$ is positive and one of the following conditions holds:
\begin{enumerate}[(A)]
\item either $p\geq n$, or
\item either $3\leq n<2p<2n$, or
\item   $1<p<n=2$ and
\begin{align*}
\alpha<\dfrac{2(p-1)^2(3p-2)}{(4-3p)^+(2-p)},
\end{align*}
\end{enumerate}
then every positive and weak solution to \eqref{eq:quasi-h} is constant.
\end{theorem}

The main idea is to make use of the divergence theorem to prove an integral growth estimate. In order to obtain the Liouville theorem on noncompact Riemannian manifold, we used the same idea as above. That is, we will use the divergence theorem together with the crucial pointwise differential identity \eqref{eq:crucial-0}. Since the manifold $M$ is noncompact, the divergence theorem holds only for  tangent vector fields which are supported in some compacted subsets. Thus, we will use the distance cutoff function.
Notice that Serrin and Zou \cite{SerZou02cauchy} used the crucial estimates of the Hessian  cutoff function  $\zeta$ which satisfies
\begin{align*}
\zeta\in C_0^{\infty}\left(B_{R}\right),\quad 0\leq \zeta\leq 1,\quad \zeta\vert_{B_{R/2}}=1,\quad \abs{\nabla\zeta}\leq\dfrac{C}{R},\quad\abs{\nabla^2\zeta}\leq\dfrac{C}{R^2}
\end{align*}
in the Euclidean space. However, it is recognized that the lower bound of the Ricci curvature generally does not constrain the Hessian of the distance function, in other words, such a cutoff function $\zeta$ appears non existent for any $R>0$  in general Riemannian manifolds characterized by nonnegative Ricci curvature.  So we can not use the method of Gidas and Spruck \cite{GidSpr81global} or Serrin and Zou \cite{SerZou02cauchy} directly. Fortunately, to obtain the Liouville theorem, one does not need any information of the Hessian of the distance function. However, one needs the information of the lower bound of the Ricci curvature to control the growth of volume. In the semilinear case where $p=2$, Ciraolo, Farina, and Polva \cite[Theorem 1.4]{CirFarPol24classification} effectively employed the distance cutoff function to provide an alternative proof of the findings by Gidas and Spruck \cite[Theorem 1.2]{GidSpr81global}.

If we assume additionally that $f$ satisfies some polynomial growth conditions, then we can prove

\begin{theorem}\label{thm:main3}
Let $M^n$ be a complete Riemannian manifold with dimension $n$ and nonnegative Ricci curvature. Assume $p>1$ and \eqref{eq:conditions-f} holds for some $0<\alpha<p_s$. If we assume additionally that for some constant $p<p_0$ such that
\begin{align}\label{eq:condition-f1}
\liminf_{t\to+\infty}t^{1-p_0}\abs{f(t)}>0,
\end{align}
and one of the following conditions holds:
\begin{enumerate}[(i)]
\item either $f$ is nonnegative, or
\item $p>\frac{n}{2}$ and
\begin{align*}
\alpha\leq\dfrac{2((n+1)p-n)(p-1)}{\left((n+1)p-2n\right)^+},
\end{align*}
\end{enumerate}
then every positive and weak solution  to \eqref{eq:quasi-h} is constant.
\end{theorem}

As an immediately consequence, we have the following generalization of Serrin and Zou \cite[Theorem II (b, c)] {SerZou02cauchy} when $M^n=\mathbb{R}^n$ and $1<p<n$.

\begin{theorem}\label{thm:main4}
Let $M^n$ be a complete and noncompact Riemannian manifold with dimension $n$ and nonnegative Ricci curvature. Assume $f$ is subcrtical with exponent $\alpha\in\left(0,\, p_s\right)$. If we assume additionally that $f$ is nonnegative and one of the following conditions holds:
\begin{enumerate}[(I)]
\item either $p\geq n$, or
\item $3\leq n<2p$, or
\item $n=2$ and $p\geq\frac{1+\sqrt{17}}{4}$, or
\item for some constant $p_0>p$ such that
\begin{align*}
\liminf_{t\to+\infty}t^{1-p_0}f(t)>0,
\end{align*}
\end{enumerate}
then  every positive and weak solution $u$ to \eqref{eq:quasi-h} is constant.
\end{theorem}

\begin{proof}
It suffices to consider the case (\Rmn{3}). In this case, we may have $\frac{1+\sqrt{17}}{4}\leq p<\frac{4}{3}$ and get
\begin{align*}
\dfrac{2(p-1)^2(3p-2)}{(4-3p)^+(2-p)}=\dfrac{2(p-1)^2(3p-2)}{(4-3p)(2-p)}\geq\dfrac{3p-2}{2-p}>\alpha.
\end{align*}
Then \autoref{thm:main2} is applicable.
\end{proof}

Besides, as a by-product we also consider the equation as follows
\begin{align*}
\Delta u-\lambda u+u^q=0,
\end{align*}
where $q>1$ and $\lambda>0$,
and provide a new and simple proof of Bidaut-Veron and Veron's Liouville result \cite[Theorem 6.1]{BidMarVer91nonlinear} on a compact manifold (see \autoref{thm:BV}).

Our paper is organized as follows: In Section 3, we establish some formulas and inequalities on some tangent vector fields, which will be used later, and then establish some important lemmas on $u$ to \eqref{eq:quasi-h}, which will play a key role in the discussion on Liouville properties. In Section 4, we prove some important $\log$-gradient estimates for positive solutions. In Section 5, we discuss the Liouville properties of the positive solutions to the Lane-Emdem-Fowler and give some necessary proofs of the related results, which is the main body of this paper.

\section{Basic formulas and inequalities}
Given a positive and $C^{\infty}$-function $u$ and constant $b$,  we consider the following vector fields
\begin{align*}
\mathbf{U}=\abs{\nabla u}^{p-2}\nabla u,\quad \mathbf{X}=u^{b}\mathbf{U}.
\end{align*}
which are defined in $\{\nabla u\neq 0\}$.
By definition
\begin{align*}
\Delta_pu=\Div\mathbf{U}.
\end{align*}
Introduce the operator $A_u$ by
\begin{align*}
A_{u}(X)=X+(p-2)\abs{\nabla u}^{-2}\hin{X}{\nabla u}\nabla u,\quad\forall X\in TM.
\end{align*}
We have
\begin{align}\label{eq:ww}
\nabla_{\mathbf{U}}\mathbf{U}=\dfrac1p\abs{\nabla u}^{p-2}A_{u}\left(\nabla\abs{\nabla u}^p\right).
\end{align}
This is a direct computation. In fact, it is easy to verify
$$\nabla_{\mathbf{U}}\nabla u=\abs{\nabla u}^{p-2}\nabla _{\nabla u}\nabla u = \frac{1}{2}|\nabla u|^{p-2}\nabla \abs{\nabla u}^2=\frac{1}{p}\nabla \abs{\nabla u}^p.$$
Hence, we have
\begin{align*}
\nabla_{\mathbf{U}}\mathbf{U}=&\abs{\nabla u}^{p-2}\nabla_{\mathbf{U}}\nabla u+(p-2)\abs{\nabla u}^{p-4}\hin{\nabla_{\mathbf{U}}\nabla u}{\nabla u}\nabla u\\
=&\abs{\nabla u}^{p-2}A_{u}\left(\nabla_{\mathbf{U}}\nabla u\right)\\
=&\dfrac1p\abs{\nabla u}^{p-2}A_{u}\left(\nabla\abs{\nabla u}^p\right).
\end{align*}
Recall the linearization of the $p$-Laplacian at $u$ which is defined by
\begin{align*}
\mathcal{L}_{p,u}\psi=\Div\left(\abs{\nabla u}^{p-2}A_{u}\left(\nabla\psi\right)\right),\quad\forall\psi\in C^{\infty}\left(M\right).
\end{align*}
Then, \eqref{eq:ww} yields
\begin{align*}
\dfrac{1}{p}\mathcal{L}_{p,u}\abs{\nabla u}^p=\Div(\nabla_{\mathbf{U}}\mathbf{U}).
\end{align*}

We need the following Bochner formula.

\begin{lem}For every tangent vector field $X\in\Gamma\left(TM\right)$, we denote by $\mathcal X=\nabla X\in\mathrm{End}(TM)$, we have the following Bochner formula
\begin{align}\label{eq:bochner-X}
\Div\left(\nabla_{X}X\right)=&\tr\left(\mathcal X^2\right)+\hin{\nabla\Div X}{X}+\mathrm{Ric}\left(X,X\right),
\end{align}
where in a local orthonormal frames $\{e_i, \ldots, e_n\}$ of $TM$,
\begin{align*}
\tr\left(\mathcal X^2\right)=\tr(\mathcal X\circ \mathcal X)=\langle \mathcal X\circ \mathcal X(e_i), e_i\rangle=\langle  \mathcal X(\nabla_{e_i}X), e_i\rangle =\langle  \nabla_{\nabla_{e_i}X}X, e_i\rangle,\quad \mathrm{div}X = \langle\nabla_{e_i}X, e_i\rangle.
\end{align*}
In particular, for every positive and smooth function $\phi$, we have the following Bochner formula
\begin{align}\label{eq:bochner-p}
\dfrac1p\mathcal{L}_{p,\phi}\abs{\nabla \phi}^p= &\abs{\nabla\phi}^{2p-4}\hin{A_{\phi}\left(\nabla_{e_i}\nabla\phi\right)}{e_j} \hin{A_{\phi}\left(\nabla_{e_j}\nabla\phi\right)}{e_i}+\hin{\nabla\Delta_p\phi}{\abs{\nabla\phi}^{p-2}\nabla\phi}\nonumber\\ &+\mathrm{Ric}\left(\abs{\nabla\phi}^{p-2}\nabla\phi,\abs{\nabla\phi}^{p-2}\nabla\phi\right),
\end{align}
provided $\nabla\phi\neq0$.
\end{lem}

\begin{proof}
This is a direct verification. In fact, we choose a local orthonormal tangent frames $\set{e_i}$ with the properties $\nabla e_{i}=0$ for all $i$ at a considering point. Then we compute at the considering point by using the Ricci identity
\begin{align*}
\Div\nabla_XX=&\hin{\nabla_{e_i}\nabla_XX}{e_i}\\
=&\hin{R\left(e_i,X\right)X+\nabla_{X}\nabla_{e_i}X+\nabla_{[e_i,X]}X}{e_i}\\
=&\mathrm{Ric}\left(X,X\right)+X\hin{\nabla_{e_i}X}{e_i}+ \hin{\nabla_{\nabla_{e_i}X}X}{e_i}\\
=&\mathrm{Ric}\left(X,X\right)+\hin{\nabla\Div X}{X}+ \tr\left(\mathcal X^2\right).
\end{align*}
Thus, we obtain the desired formula.

One readily sees from \eqref{eq:ww},
\begin{align*}
\dfrac{1}{p}\mathcal{L}_{p,\phi}\abs{\nabla\phi}^p=\Div(\nabla_{\mathbf{\Phi}}\mathbf{\Phi}),
\end{align*}
where
$$\mathbf{\Phi}=\abs{\nabla\phi}^{p-2}\nabla\phi.$$
Applying the Bochner formula \eqref{eq:bochner-X}, we obtain
\begin{align}\label{eq:bochner-phi-1}
\dfrac{1}{p}\mathcal{L}_{p,\phi}\abs{\nabla\phi}^p=\hin{\nabla_{\nabla_{e_i}\mathbf{\Phi}}\mathbf{\Phi}}{e_i}+\hin{\Div\mathbf{\Phi}}{\mathbf{\Phi}}
+\mathrm{Ric}\left(\mathbf{\Phi},\mathbf{\Phi}\right).
\end{align}
Notice that $\Div\mathbf{\Phi}=\Delta_p\phi$ from \eqref{eq:ww}. It suffices to compute the first terms on the right hand side of \eqref{eq:bochner-phi-1}. In fact, a similar computation as in \eqref{eq:ww}, we have
\begin{align*}
\nabla_{e_i}\mathbf{\Phi}=\abs{\nabla \phi}^{p-2}A_{\phi}\left(\nabla_{e_i}\nabla\phi\right).
\end{align*}
We complete the proof and obtain the desired Bochner formula \eqref{eq:bochner-p}.
\end{proof}

From now on, we always assume that $u\in W_{loc}^{1,p}\left(\Omega\right)\cap L_{loc}^{\infty}\left(\Omega\right)$ is a weak and positive solution of the equation \eqref{eq:quasi-h} with  $f\in C^0([0,\infty))\cap C^1((0,\infty))$ in $\Omega\subset M$.
We denote
$$
\Omega_{cr} = \{x\in \Omega:\nabla u(x)=0\}.
$$
According to Theorem 1.4 in \cite{Antonini2023} and the classical regularity theory (for example, see \cite{DiBenedetto, T1984,Uhl77regularity}), we know that
$$u\in C_{loc}^{1,\beta}(\Omega)\cap W^{2,2}_{loc}(\Omega\setminus\Omega_{cr})\quad \mbox{and}\quad u\in C_{loc}^{\infty}(\Omega^c_{cr}).$$
For the vector fields
$$\mathbf{U}=\abs{\nabla u}^{p-2}\nabla u\quad \mbox{and} \quad \mathbf{X}=u^{b}\mathbf{U}$$
defined as before, we have $\mathbf{U}$, $\mathbf{X}\in C_{loc}^{\alpha}$ and $\mathbf{U}=\mathbf{X}=0$ on $\Omega_{cr}$ since $p>1$.

Next, we need to consider that
$$\mathcal U=\nabla \mathbf U\quad \mbox{and}\quad \mathcal X=\nabla\mathbf X\in\mathrm{End}(TM).$$
It is worthy to point out that the endomorphisms $\mathcal{U}$ and $\mathcal{X}$ may have singularities on $\Omega_{cr}$. But we have the following conclusions on regularity which is implied in \cite{Antonini2023}. For the sake of convenience, we give a proof here.

\begin{lem}\label{lem:regular}
Let $u\in W_{loc}^{1,p}\left(\Omega\right)\cap L_{loc}^{\infty}\left(\Omega\right)$ be a weak and positive solution $u$ of \eqref{eq:quasi-h}, $\mathcal U$ and $\mathcal X$ be the associated endomorphisms as above. Then, we have
$$
\mathcal{U} ~~\mbox{and}~~ \mathcal{X}\in L^2_{loc}(\Omega).
$$
\end{lem}

\begin{proof}
It is obvious that
\begin{align*}
\mathcal{U} = \abs{\nabla u}^{p-2}\left(id_{TM}+(p-2)\frac{\nabla u}{\abs{\nabla u}}\otimes\frac{d u}{\abs{\nabla u}}\right)\nabla^2u\quad\text{and}\quad \mathcal X = bu^{b-1}\mathbf U\otimes du+u^b\mathcal U.
\end{align*}
When $p\geq 2$,  we have $\mathcal{U}, \mathcal{X}\in L^2_{loc}(\Omega)$ (for example see \cite[Corollary in P. 117]{SerZou02cauchy}). Now we consider the case $1<p<2$. According to \cite[Proposition 8.1]{SerZou02cauchy}, for any compact set $K$,
\begin{align*}
\abs{\nabla u}^{p-2}\nabla^2u\in L^2\left(K\setminus \Omega_{cr}\right).
\end{align*}
By a very recent result \cite[Corollary 1.6]{Antonini2023}, the measure of critical set $\Omega_{cr}$ is zero. Since
$$id_{TM}+(p-2)\frac{\nabla u}{\abs{\nabla u}}\otimes\frac{d u}{\abs{\nabla u}}\in L_{loc}^\infty\quad\mbox{and}\quad bu^{b-1}\mathbf U\otimes du\in C_{loc}^\alpha,$$ we can see
$$
\mathcal{U} ~~\mbox{and}~~ \mathcal{X}\in L^2_{loc}(\Omega).
$$
Thus, we finish the proof.
\end{proof}

We define
\begin{align*}
\mathscr X = \mathcal{X}-\frac{\mathrm{tr}\mathcal X}{n}id_{TM}=\mathcal{X}-\frac{\mathrm{div}\mathbf X}{n}id_{TM}\in \mathrm{End}(TM).
\end{align*}
Since
$$\Div\mathbf X = bu^{b-1}\abs{\nabla u}^p-u^bf(u)\in C_{loc}^0(\Omega),$$
thus $\mathscr X\in L^2_{loc}(\Omega)$. By the Lemma due to Serrin and Zou \cite[Lemma 6.3]{SerZou02cauchy}, we know that
\begin{align*}
\mathrm{tr}\left(\mathscr X^2\right)=\mathrm{tr}(\mathcal X^2)-\frac{1}{n}(\mathrm{tr}\mathcal X)^2\geq 0.
\end{align*}
It is obvious that
\begin{align*}
\mathrm{tr}\left(\mathscr X^2\right)\leq \abs{\mathscr X}^2
\end{align*}
and the "=" establishes if and only if $\mathscr X$ is a self-dual endomorphism, i.e.,
$$\langle\mathscr X(Y),\, Z\rangle=\langle Y,\, \mathscr X(Z)\rangle,\quad \forall Y, \, Z\in \Gamma(TM).$$
Here
$$\abs{\mathscr X}^2 =  \mathscr X_i^j \mathscr X_k^lg_{jl}g^{ik}.$$

In general, $\mathrm{tr}\left(\mathscr X^2\right)$ can not control $\abs{\mathscr X}^2$. However for the special tensor $\mathscr X$ defined above, we have the following algebraic inequalities, which is extremely  important for the proof of the main theorem.

\begin{lem}\label{lem:basic-3}
For the 2-tensor $\mathscr X$ defined above, the following estimate holds a.e. on $\Omega$
\begin{align*}
\frac{(p-1)^2+1}{2(p-1)}\tr\left(\mathscr X^2\right)\geq \abs{\mathscr X}^2\geq \tr\left(\mathscr X^2\right).
\end{align*}
\end{lem}
		
\begin{proof}
Direct computation shows that
\begin{align*}
\tr\left(\mathscr X^2\right) = \left(\mV_{i}^j-\frac{1}{n}\tr\mV \delta_i^j\right)\left(\mV^{i}_j-\frac{1}{n}\tr\mV \delta_j^i\right)= \tr(\mV^2)-\frac{1}{n}\left(\tr\mV\right)^2
\end{align*}
and
\begin{align*}
\abs{\mathscr X}^2 = \left(\mV_{i}^j-\frac{\tr\mV}{n} \delta_i^j\right)\left(\mV_{i}^j-\frac{\tr\mV}{n}\delta_i^j\right) =\tr\left(\mV\mV^T\right)-\frac{1}{n}\left(\tr\mV\right)^2.
\end{align*}
			
Let $\{e_1,e_2,\ldots, e_n\}$ be  a local orthonormal frames of $TM$ on a domain with $\nabla u\neq 0$ such that $e_1=\frac{\nabla u}{\abs{\nabla u}}$. For simplicity, we denote by $h=\abs{\nabla u}^2$.
		
First, we have
\begin{align*}
\mV_{ij}=&\nabla_i\left(u^bh^{\frac{p}{2}-1}u_j\right)\\
=&bu^{b-1}h^{\frac{p}{2}-1}u_iu_j+u^b\left(h^{\frac{p}{2}-1}u_{ij}+
\left(\frac{p}{2}-1\right)h^{\frac{p}{2}-2}h_iu_j\right).
\end{align*}
It follows that
\begin{align*}
\mathrm{tr}\left(\mathcal X\mathcal X^T\right)=\sum_{i,j}\mV_{ij}^2
=&b^2u^{2b-2}h^{p} +u^{2b}h^{p-2}u_{ij}^2+\frac{p}{2}\left(\frac{p}{2}-1\right) u^{2b}h^{p-3}|\nabla h|^2
+b(p-1)u^{2b-1}h^{p-2}u_ih_i
\end{align*}
and
\begin{align*}
\mathrm{tr}\left(\mathcal X^2\right)=\sum_{i,j}\mV_{ij}\mV_{ji}=&b^2u^{2b-2}h^{p} +u^{2b}h^{p-2}u_{ij}^2+b(p-1)u^{2b-1}h^{p-2}u_ih_i\\
&+u^{2b}\left(\frac{p}{2}-1\right)^2h^{p-4}\langle\nabla h, \nabla u\rangle^2+u^{2b}\left(\frac{p}{2}-1\right)h^{p-3}|\nabla h|^2.
\end{align*}
			
Obviously, we have $u_1 = h^{1/2}$. Direct computation shows
\begin{align*}
u_{11} = \frac{1}{2}h^{-1/2}h_1 = \frac{1}{2}h^{-1}\langle\nabla u,\nabla h\rangle.
\end{align*}
Thus we have
\begin{align*}
\sum_{i,j}\mV_{ij}^2=&\left(bu^{b-1}h^{\frac{p}{2}}+(p-1)u^bh^{\frac{p}{2}-1}u_{11}\right)^2 +u^{2b}h^{p-2}\sum_{i,j\geq 2}u_{1i}^2
+\left((p-1)^2+1\right) u^{2b}h^{p-2}\sum_{i=2}^nu_{1i}^2
\end{align*}
and
\begin{align}\label{eq:linear-1}
\sum_{i,j}\mV_{ij}\mV_{ji}=&\left(bu^{b-1}h^{\frac{p}{2}}+(p-1)u^bh^{\frac{p}{2}-1}u_{11}\right)^2+u^{2b}h^{p-2}\sum_{i,j\geq2} u_{ij}^2
+2u^{2b}(p-1)h^{p-1} \sum_{i=2}^nu_{1i}^2.
\end{align}
We can also compute
\begin{align*}				
\left(\Div\mathbf{X}\right)^2 = \left(bu^{b-1}h^{\frac{p}{2} } +u^b\left(h^{\frac{p}{2}-1}\sum_{i=2}^nu_{ii}+ (p-1)h^{\frac{p}{2}-1} u_{11}\right)\right)^2.
\end{align*}
			
If we denote
\begin{align*}
J_0 =& \left(bu^{b-1}h^{\frac{p}{2}}+(p-1)u^bh^{\frac{p}{2}-1}u_{11}\right)^2+u^{2b}h^{p-2}\sum_{i,j\geq2}^nu_{ij}^2\\
&-\frac{1}{n}\left(bu^{b-1}h^{\frac{p}{2}} +u^b\left(h^{\frac{p}{2}-1}\sum_{i=2}^nu_{ii}+ (p-1)h^{\frac{p}{2}-1}u_{11}\right)\right)^2,
\end{align*}
then we have
\begin{align*}
J_0\geq& \left(bu^{b-1}h^{\frac{p}{2}}+(p-1)u^bh^{\frac{p}{2}-1}u_{11}\right)^2+u^{2b}h^{p-2}\sum_{i=2}^nu_{ii}^2\\
&-\frac{1}{n}\left(bu^{b-1}h^{\frac{p}{2} }+(p-1)u^bh^{\frac{p}{2}-1} u_{11} +u^b h^{\frac{p}{2}-1}\sum_{i=2}^nu_{ii}\right)^2\\
\geq& 0.
\end{align*}
			
Since
\begin{align*}
\abs{\mathscr X}^2 = J_0+((p-1)^2+1) u^{2b}h^{p-2}\sum_{i=2}^nu_{1i}^2
\end{align*}
and
\begin{align*}
\tr\left(\mathscr X^2\right) = J_0+2(p-1)u^{2b}h^{p-1} \sum_{i=2}^nu_{1i}^2,
\end{align*}
we obtain
\begin{align*}
\dfrac{(p-1)^2+1}{2(p-1)}\tr\left(\mathscr X^2\right)\geq \abs{\mathscr X}^2\geq\tr\left(\mathscr X^2\right).
\end{align*}
Thus, we finish the proof of this lemma.
\end{proof}

\begin{rem}
From the proof of the above Lemma \ref{lem:basic-3}, it is easy to see that for the tensor $\mathscr X$ difined as before, where $u$ is a $C^3$-smooth positive function with $|\nabla u|\neq 0$ on $\Omega$ (in other words, we do not need assume that $u$ is a solution to the equation \eqref{eq:quasi-h}), there holds pointwisely true the same inequality
\begin{align*}
\dfrac{(p-1)^2+1}{2(p-1)}\tr\left(\mathscr X^2\right)\geq \abs{\mathscr X}^2\geq\tr\left(\mathscr X^2\right).
\end{align*}
\end{rem}

Based on the Bochner formula \eqref{eq:bochner-X}, we can state the following general pointwise identity.
\begin{lem}\label{lem:basic1}
For any positive solution $u$ of \eqref{eq:quasi-h} in $\Omega$, constants $a$ and $b$,  the following identity
\begin{align*}
&\Div\left(u^a\left(\nabla_\mathbf{X}\mathbf{X}-\Div \mathbf{X} \mathbf{X}\right)-\dfrac{p-1}{p}\hin{\mathbf{X}}{\nabla u^{a}}\mathbf{X}\right)\\
=&u^a\left(\mathrm{tr}(\mathcal X^2)-\abs{\Div
\mathbf{X}}^2+\mathrm{Ric}\left(\mathbf{X},\mathbf{X}\right)\right)-\dfrac{(p-1)(a+2b-1)a}{p}u^{a+2b-2}\abs{\nabla u}^{2p}-\dfrac{(2p-1)a}{p}u^{a+2b-1}\abs{\nabla u}^p\Delta_pu
\end{align*}
holds a.e. on $\Omega$.
\end{lem}

\begin{proof}
We compute from \eqref{eq:bochner-X} that
\begin{align}
&\Div\left(u^a\left(\nabla_\mathbf{X}\mathbf{X}-\Div\mathbf{X} \mathbf{X}\right)\right)\notag\\
=&u^a\Div\left(\nabla_\mathbf{X}\mathbf{X}-\Div\mathbf{X} \mathbf{X}\right)+\hin{\nabla u^a}{\nabla_{\mathbf{X} }\mathbf{X} -\Div\mathbf{X} \mathbf{X}}\notag\\
=&u^a\left(\mathrm{Ric}\left(\mathbf{X},\mathbf{X}\right)+\hin{\nabla\Div\mathbf{X}}{\mathbf{X}}+\mathrm{tr}(\mathcal X^2)-\left(\hin{\mathbf{X}}{\nabla\Div\mathbf{X}}+\abs{\Div \mathbf{X}}^2\right)\right)\notag\\
&+au^{a+2b-1}\abs{\nabla u}^{2p-4}\hin{\nabla u}{\nabla_{\nabla u}\nabla u-\Delta u\nabla u}\notag\\
=&u^a\left(\mathrm{Ric}\left(\mathbf{X},\mathbf{X}\right)+\mathrm{tr}(\mathcal X^2)-\abs{\Div \mathbf{X}}^2\right)+(p-1)au^{a+2b-1}\abs{\nabla u}^{2p-4}\Delta_{\infty}u-au^{a+2b-1}\abs{\nabla u}^{p}\Delta_pu.\label{eq:basic1}
\end{align}
On the other hand,
\begin{align}
\Div\left(u^{a-1}\hin{\mathbf{X}}{\nabla u}\mathbf{X}\right)=&\Div\left(u^{a+2b-1}\abs{\nabla u}^p\mathbf{U}\right)\notag\\
=&\hin{\nabla\left(u^{a+2b-1}\abs{\nabla u}^p\right)}{\mathbf{U}}+u^{a+2b-1}\abs{\nabla u}^p\Div\mathbf{U}\notag\\
=&(a+2b-1)u^{a+2b-2}\abs{\nabla u}^{2p}+pu^{a+2b-1}\abs{\nabla u}^{2p-4}\Delta_{\infty}u+u^{a+2b-1}\abs{\nabla u}^p\Delta_pu.\label{eq:basic2}
\end{align}
We obtain from \eqref{eq:basic1} and \eqref{eq:basic2}
\begin{align*}
&\Div\left(u^a\left(\nabla_\mathbf{X}\mathbf{X}-\Div\mathbf{X}\mathbf{X}\right)-\dfrac{p-1}{p}\hin{\mathbf{X}}{\nabla u^{a}}\mathbf{X}\right)\\
=&u^a\left(\mathrm{Ric}\left(\mathbf{X},\mathbf{X}\right)+\mathrm{tr}(\mathcal X^2)-\abs{\Div \mathbf{X}}^2\right)+(p-1)au^{a+2b-1}\abs{\nabla u}^{2p-4}\Delta_{\infty}u-au^{a+2b-1}\abs{\nabla u}^{p}\Delta_pu\\
&-\dfrac{(p-1)a}{p}\left((a+2b-1)u^{a+2b-2}\abs{\nabla u}^{2p}+pu^{a+2b-1}\abs{\nabla u}^{2p-4}\Delta_{\infty}u+u^{a+2b-1}\abs{\nabla u}^p\Delta_pu\right)\\
=&u^a\left(\mathrm{tr}(\mathcal X^2)-\abs{\Div \mathbf{X}}^2+\mathrm{Ric}\left(\mathbf{X},\mathbf{X}\right)\right)-\dfrac{(p-1)(a+2b-1)a}{p}u^{a+2b-2}\abs{\nabla u}^{2p}-\dfrac{(2p-1)a}{p}u^{a+2b-1}\abs{\nabla u}^p\Delta_pu.
\end{align*}
We complete the proof.
\end{proof}

It is worthy to point out that the above lemma holds true for any $u\in C^3(\Omega)$ with $\nabla u\neq 0$ a.e. on $\Omega$ and one does not need to assume that $u$ satisfies the equation \eqref{eq:quasi-h}. Applying this lemma we can derive the following fundamental and important pointwise identity.

\begin{lem}\label{lem:basic2}
Let $u$ be a weak and positive solution to \eqref{eq:quasi-h} in $\Omega\subset M$.  For any $a$ and $b\in\mathbb{R}$, the following identity holds in the $L^1_{loc}$ sense,
\begin{align}
\label{eq:basic-0}
\begin{split}
&\Div\left(u^{a}\left(\mathscr{X}(\mathbf X)\right)-\dfrac{n(p-1)a+(n-1)pb}{np}u^{a-1}\hin{\mathbf{X}}{\nabla u}\mathbf{X}\right)\\
=&u^{a}\left(\tr\left(\mathscr X^2\right)+\mathrm{Ric}\left(\mathbf{X},\mathbf{X}\right)\right)\\
&+\left(\dfrac{(p-1)a}{p}\left(1-\dfrac{(n-p)a}{(n-1)p}\right)-\dfrac{n-1}{n}\left(b+\dfrac{n(p-1)a}{(n-1)p}\right)^2\right)u^{a+2b-2}
\abs{\nabla u}^{2p}\\
&+\dfrac{n-1}{n}\left(\dfrac{((n+1)p-n)a}{(n-1)p}f(u)-uf'(u)\right)u^{a+2b-1}\abs{\nabla u}^{p}.
\end{split}
\end{align}
Here $$\mathbf{X}=u^{b}\abs{\nabla u}^{p-2}\nabla u.$$
In particular, one takes
\begin{align*}
a=\dfrac{(n-1)p\alpha}{(n+1)p-n},\quad\quad b=-\dfrac{n(p-1)\alpha}{(n+1)p-n},
\end{align*}
and
$$\mathbf{X}=u^{-\frac{n(p-1)\alpha}{(n+1)p-n}}\abs{\nabla u}^{p-2}\nabla u$$
to obtain the following identity
\begin{align}\label{eq:basic'}
\begin{split}&\Div\left(u^{\frac{(n-1)p\alpha}{(n+1)p-n}}\left(\mathscr{X}(\mathbf X)\right)\right)\\
=\,&u^{\frac{(n-1)p\alpha}{(n+1)p-n}}\left(\tr\left(\mathscr X^2\right)+\mathrm{Ric}\left(\mathbf{X},\mathbf{X}\right)\right)\\
&+\dfrac{(n-1)(p-1)\alpha}{(n+1)p-n}\left(1-\dfrac{(n-p)\alpha}{(n+1)p-n}\right)u^{-\frac{((n+1)p-2n)\alpha}{(n+1)p-n}-2}\abs{\nabla u}^{2p}\\
&+\dfrac{n-1}{n}\left(\alpha f(u)-uf'(u)\right)u^{-\frac{((n+1)p-2n)\alpha}{(n+1)p-n}-1}\abs{\nabla u}^{p}.
\end{split}
\end{align}
\end{lem}

\begin{proof}
One can check
\begin{align*}
\Div\mathbf{X}=u^b\Delta_pu+bu^{b-1}\abs{\nabla u}^p.
\end{align*}
According to \autoref{lem:basic1},
\begin{align}
&\Div\left(u^a\left(\nabla_\mathbf{X}\mathbf{X}-\Div\mathbf{X}\mathbf{X}\right)-\dfrac{p-1}{p}\hin{\mathbf{X}}{\nabla u^{a}}\mathbf{X}\right)\notag\\
=&u^{a}\left(\hin{\nabla_{\nabla_{e_i}\mathbf{X}}\mathbf{X}}{e_i}-\dfrac{1}{n}\abs{\Div \mathbf{X}}^2+\mathrm{Ric}\left(\mathbf{X},\mathbf{X}\right)\right)\notag\\
&-\dfrac{n-1}{n}u^a\abs{\Div \mathbf{X}}^2-\dfrac{(p-1)(a+2b-1)a}{p}u^{a+2b-2}\abs{\nabla u}^{2p}-\dfrac{(2p-1)a}{p}u^{a+2b-1}\abs{\nabla u}^p\Delta_pu\notag\\
=&u^{a}\left(\tr\left(\mathscr X^2\right)+\mathrm{Ric}\left(\mathbf{X},\mathbf{X}\right)\right)-\dfrac{n-1}{n}u^a\abs{u^{b}\Delta_pu+bu^{b-1}\abs{\nabla u}^p}^2\notag\\
&-\dfrac{(p-1)(a+2b-1)a}{p}u^{a+2b-2}\abs{\nabla u}^{2p}-\dfrac{(2p-1)a}{p}u^{a+2b-1}\abs{\nabla u}^p\Delta_pu \notag\\
=&u^{a}\left(\tr\left(\mathscr X^2\right)+\mathrm{Ric}\left(\mathbf{X},\mathbf{X}\right)\right)\notag\\
&+\left(\dfrac{p-1}{p}a-\dfrac{n-1}{n}b^2-\dfrac{p-1}{p}a^2-\dfrac{2(p-1)}{p}ab\right)u^{a+2b-2}\abs{\nabla u}^{2p}-\dfrac{n-1}{n}u^{a+2b}\left(\Delta_pu\right)^2\notag\\
&-\left(\dfrac{2(n-1)b}{n}+\dfrac{(2p-1)a}{p}\right)u^{a+2b-1}\abs{\nabla u}^{p}\Delta_pu.\label{eq:basic3}
\end{align}

By the equation \eqref{eq:quasi-h}, we have
\begin{align*}
u^{a+2b}\left(\Delta_pu\right)^2=&u^{a+2b}\Delta_pu\Div\mathbf{U}\\
=\,&\Div\left(u^{a+2b}\Delta_pu\mathbf{U}\right)-\hin{\nabla\left(u^{a+2b}\Delta_pu\right)}{\mathbf{U}}\\
=&\Div\left(u^{a+2b}\Delta_pu\abs{\nabla u}^{p-2}\nabla u\right)-\left((a+2b)u^{a+2b-1}\abs{\nabla u}^{p}\Delta_pu+u^{a+2b}\abs{\nabla u}^{p-2}\hin{\nabla\Delta_pu}{\nabla u}\right)\\
=&\Div\left(u^{a}\left(\Div\mathbf{X}-bu^{-1}\hin{\mathbf{X}}{\nabla u}\right) \mathbf{X}\right)+\left((a+2b)f(u)+uf'(u)\right)u^{a+2b-1}\abs{\nabla u}^{p}.
\end{align*}
Inserting the above identities into \eqref{eq:basic3}, we obtain
\begin{align*}
&\Div\left(u^{a}\left(\nabla_{\mathbf{X}}\mathbf{X}-\dfrac{\Div\mathbf{X}}{n}\mathbf{X}\right)-\dfrac{n(p-1)a+(n-1)pb}{np}u^{a-1}
\hin{\mathbf{X}}{\nabla u}\mathbf{X}\right)\\
=\,&u^{a}\left(\tr\left(\mathscr X^2\right)+\mathrm{Ric}\left(\mathbf{X},\mathbf{X}\right)\right)\\
&+\left(\dfrac{p-1}{p}a-\dfrac{n-1}{n}b^2-\dfrac{p-1}{p}a^2-\dfrac{2(p-1)}{p}ab\right)u^{a+2b-2}\abs{\nabla u}^{2p}\\
&+\left(\dfrac{((n+1)p-n)a}{np}f(u)-\dfrac{n-1}{n}uf'(u)\right)u^{a+2b-1}\abs{\nabla u}^{p}.
\end{align*}
Finally, noticing that
\begin{align*}
\dfrac{p-1}{p}a-\dfrac{n-1}{n}b^2-\dfrac{p-1}{p}a^2-\dfrac{2(p-1)}{p}ab=&\dfrac{(p-1)a}{p}\left(1-\dfrac{(n-p)a}{(n-1)p}\right)
-\dfrac{n-1}{n}\left(b+\dfrac{n(p-1)a}{(n-1)p}\right)^2,
\end{align*}
we get the desired formula \eqref{eq:basic'} and complete the proof formally. By \autoref{lem:regular}, we know that $\tr\left(\mathcal X^2\right)\in L^1_{loc}$.  It follows that every term on the right hand side of \cref{lem:basic2} lies in $L^1_{loc}$.
\end{proof}

In the end of this section, we collect some useful inequalities in geometric analysis. Let $M^n$ be a Riemannian manifold with Ricci curvature bounded from below by some nonnegative constant $-(n-1)\kappa$. The first one is the following Saloff-Coste's Dirichlet Sobolev inequality: for every $\psi\in C_0^{\infty}\left(B_{R}\right)$
\begin{align*}
\left(\fint_{B_{R}}\abs{\psi}^{\frac{n}{n-1}}\right)^{\frac{n-1}{n}}\leq e^{C_{n}\left(1+\sqrt{\kappa}R\right)}R\left(\fint_{B_{R}}\abs{\nabla\psi}+R^{-1}\fint_{B_{R}}\abs{\psi}\right).
\end{align*}
We also have the following Dirichlet Poincar\'e inequality
\begin{align*}
\int_{B_R}\abs{\psi}\leq e^{C_n\left(1+\sqrt{\kappa}R\right)}R\int_{B_R}\abs{\nabla\psi},\quad \forall \psi\in C_0^{\infty}\left(B_{R}\right).
\end{align*}
Hence, we have the following Sobolev inequality
\begin{align*}
\left(\fint_{B_{R}}\abs{\psi}^{\frac{n}{n-1}}\right)^{\frac{n-1}{n}}\leq e^{C_{n}\left(1+\sqrt{\kappa}R\right)}R\fint_{B_{R}}\abs{\nabla\psi}, \quad\forall \psi\in C_0^{\infty}\left(B_{R}\right).
\end{align*}
Denote
\begin{align*}
C_{\mathcal{SD},1}=e^{C_{n}\left(1+\sqrt{\kappa}R\right)},
\end{align*}
and set
\begin{align*}
\chi_{n,p}=\begin{cases}
\frac{n}{n-p},&1\leq p<n,\\
>\frac{n}{(n-1)p},&p\geq n.
\end{cases}
\end{align*}
Then for $p\in(1,\infty)$
\begin{align*}
\left(\fint_{B_{R}}\abs{\psi}^{p\chi_{n,p}}\right)^{\frac{n-1}{n}}\leq&C_{\mathcal{SD},1}R\fint_{B_R}\abs{\nabla\abs{\psi}^{\frac{(n-1)p\chi_{n,p}}{n}}}\\
\leq&\dfrac{(n-1)p\chi_{n,p}}{n}C_{\mathcal{SD},1}R\fint_{B_R}\abs{\psi}^{\frac{(n-1)p\chi_{n,p}}{n}-1}\abs{\nabla\psi}\\
\leq&\dfrac{(n-1)p\chi_{n,p}}{n}C_{\mathcal{SD},1}R\left(\fint_{B_R}\abs{\nabla\psi}^p\right)^{\frac{1}{p}}
\left(\fint_{B_R}\abs{\psi}^{\left(\frac{(n-1)p\chi_{n,p}}{n}-1\right)\frac{p}{p-1}}\right)^{\frac{p-1}{p}}.
\end{align*}
By the definition, we know that
\begin{align*}
0<\left(\dfrac{(n-1)p\chi_{n,p}}{n}-1\right)\dfrac{p}{p-1}=\begin{cases}
\chi_{n,p},&1\leq p<n,\\
<\chi_{n,p},&p\geq n.
\end{cases}
\end{align*}
Hence
\begin{align*}
\left(\fint_{B_{R}}\abs{\psi}^{p\chi_{n,p}}\right)^{\frac{n-1}{n}}
\leq&\dfrac{(n-1)p\chi_{n,p}}{n}C_{\mathcal{SD},1}R\left(\fint_{B_R}\abs{\nabla\psi}^p\right)^{\frac{1}{p}}
\left(\fint_{B_R}\abs{\psi}^{p\chi_{n,p}}\right)^{\frac{(n-1)p\chi_{n,p}-n}{np\chi_{n,p}}},
\end{align*}
which implies
\begin{align*}
\left(\fint_{B_{R}}\abs{\psi}^{p\chi_{n,p}}\right)^{\frac{1}{p\chi_{n,p}}}
\leq&C_{\mathcal{SD},p}R\left(\fint_{B_R}\abs{\nabla\psi}^p\right)^{\frac{1}{p}},
\end{align*}
where
\begin{align*}
C_{\mathcal{SD},p}=\dfrac{(n-1)p\chi_{n,p}}{n}C_{\mathcal{SD},1}.
\end{align*}

Similarly, for the Dirichlet Poincar\'e inequalities, we have
\begin{align}\label{eq:sobolev-p}
\left(\int_{B_R}\abs{\psi}^p\right)^{\frac{1}{p}}\leq C_{\mathcal{PD},p}R\left(\int_{B_R}\abs{\nabla\psi}^p\right)^{1/p},\quad\forall\psi\in C_0^{\infty}\left(B_R\right),
\end{align}
where
\begin{align*}
C_{\mathcal{PD},p}=pC_{\mathcal{PD},1}.
\end{align*}
In particular, we have the following Sobolev inequality
\begin{align}\label{eq:sobolev}
\left(\fint_{B_{R}}\abs{\psi}^{2\chi}\right)^{\frac{1}{\chi}}\leq e^{C_{n}\left(1+\sqrt{\kappa}R\right)}R^2\fint_{B_{R}}\abs{\nabla\psi}^2, \quad \forall \psi\in C_0^{\infty}\left(B_{R}\right).
\end{align}
Here the constant $\chi=\frac{n}{n-2}$ if $n>2$, and $\chi$ can be any positive number which is larger than $1$. We also have the following Neumann Poincar\'e inequalities
\begin{align*}
\int_{B_R}\abs{\psi-\fint_{B_{R}}\psi}^p\leq e^{C_{n,p}\left(1+\sqrt{\kappa}R\right)}R^p\int_{B_R}\abs{\nabla\psi}^p,\quad\forall\psi\in C^{\infty}\left(B_R\right).
\end{align*}

\section{Gradient estimates}
Throughout this section, we always assume the condition \eqref{eq:conditions-f2} holds. That is, we always assume for some constant $\delta_0<1$, there holds
\begin{align*}
\dfrac{p-1}{n-1}\left((n+1)f(t)+2\delta_0\abs{f(t)}\right)-tf'(t)\geq0,\quad\forall t>0.
\end{align*}
The main result of this section is the global gradient estimate \autoref{thm:global-gradient-estimate}:
\begin{quotation}
Let $M^n$ be a complete and $n$-dimensional Riemannian manifold with the Ricci curvature bounded from below by a nonpositive constant $-(n-1)\kappa$. Assume for some constant $\delta_0<1$,
\begin{align*}
\dfrac{p-1}{n-1}\left((n+1)f(t)+2\delta_0\abs{f(t)}\right)-tf'(t)\geq0,\quad\forall t>0.
\end{align*}
If $p>1$ and $u$ is a positive and weak solution to
\begin{align*}
-\Delta_pu=f(u),
\end{align*}
then
\begin{align*}
\norm{\nabla\ln u}_{L^{\infty}\left(M\right)}\leq \dfrac{n-1}{p-1}\sqrt{\dfrac{\kappa}{1-\delta_0^+}}.
\end{align*}
\end{quotation}

We divided the proof into several steps.

Firstly, we state a pointwise differential inequality.

\begin{lem}\label{lem:moser}
Let $M^n$ be a complete and $n$-dimensional Riemannian manifold with the Ricci curvature bounded from below by a nonpositive constant $-(n-1)\kappa$. Assume $p>1$ and the condition \eqref{eq:conditions-f2} holds. If $u$ is a positive and weak solution to \eqref{eq:quasi-h}, then the following pointwise inequality holds for every $\lambda\geq\lambda_0\coloneqq\frac{p}{1-\delta_0^+}$
\begin{align}\label{eq:moser}
\dfrac{1}{\lambda}\mathcal{L}_{p,\ln u}\abs{\nabla\ln u}^{\lambda }\geq&\left(1-\delta_0^+\right)\dfrac{(p-1)^2}{n-1}\abs{\nabla\ln u}^{\lambda+p}-(n-1)\kappa \abs{\nabla\ln u}^{\lambda+p-2}-p(p-1)\abs{\nabla\ln u}^{\lambda+p-2}\abs{\nabla\abs{\nabla\ln u}},
\end{align}
provided $\set{\nabla u\neq0}$.
\end{lem}

\begin{proof}
It is well known that $u$ is $C^{\infty}$ in $\set{\nabla u\neq0}$ (see \cite{DiBenedetto, T1984}). We consider new functions
$$w=\ln u\quad \mbox{and} \quad F=\abs{\nabla w}^{p}.$$
Consider the following vector field
\begin{align*}
\mathbf{W}=\abs{\nabla w}^{p-2}\nabla w
\end{align*}
which is  defined in $\set{\nabla u\neq0}$. Applying the Bochner formula \eqref{eq:bochner-p} to the function $w$, we have
\begin{align}\label{eq:bochner-w}
\dfrac{1}{p}\mathcal{L}_{p,w}\abs{\nabla w}^p=\hin{\nabla_{\nabla_{e_i}\mathbf{W}}\mathbf{W}}{e_i}+\hin{\nabla\Delta_pw}{\mathbf{W}}+\mathrm{Ric}\left(\mathbf{W},\mathbf{W}\right).
\end{align}

We claim that the following Kato inequality holds true
\begin{align}\label{eq:kato-w}
\hin{\nabla_{\nabla_{e_i}\mathbf{W}}\mathbf{W}}{e_i}\geq&\dfrac{1}{n-1}\abs{\Delta_pw}^2-\dfrac{2}{(n-1)p}F^{-1}\hin{A_w\left(\nabla F\right)}{\mathbf{W}}\Delta_pw+\dfrac{n(p-1)}{(n-1)p^2}F^{-2/p}\hin{A_w\left(\nabla F\right)}{\nabla F}.
\end{align}
To see this, assume $\nabla w(x_0)\neq0$. Choose a local orthonormal tangent frames $\set{e_i}$ such that at $x_0$
\begin{align*}
\nabla w=w_1e_1.
\end{align*}
Then at the considering point $x_0$, it follows from \eqref{eq:linear-1}
\begin{align*}
\hin{\nabla_{\nabla_{e_i}\mathbf{W}}\mathbf{W}}{e_i}=&\left((p-1)w_1^{p-2}w_{11}\right)^2+2(p-1)\sum_{j=2}^nw_1^{2p-4}w_{1j}^2
+\sum_{i,j=2}^nw_1^{2p-4}w_{ij}^2\\
\geq&\left((p-1)w_1^{p-2}w_{11}\right)^2+2(p-1)\sum_{j=2}^nw_1^{2p-4}w_{1j}^2+\dfrac{1}{n-1}\sum_{j=2}^n\left(w_1^{p-2}w_{jj}\right)^2\\
=&\left((p-1)w_1^{p-2}w_{11}\right)^2+2(p-1)\sum_{j=2}^nw_1^{2p-4}w_{1j}^2+\dfrac{1}{n-1}\sum_{j=2}^n\left(\Delta_pw-(p-1)w_1^{p-2}w_{11}\right)^2\\
=&\dfrac{1}{n-1}\abs{\Delta_pw}^2+2(p-1)\sum_{j=2}^nw_1^{2p-4}w_{1j}^2+\dfrac{n(p-1)^2}{n-1}w_{1}^{2p-4}w_{11}^2
-\dfrac{2(p-1)}{n-1}w_1^{p-2}w_{11}\Delta_pw\\
\geq&\dfrac{1}{n-1}\abs{\Delta_pw}^2+\dfrac{n}{n-1}\left((p-1)\sum_{j=2}^nw_1^{2p-4}w_{1j}^2+(p-1)^2w_{1}^{2p-4}w_{11}^2\right)
-\dfrac{2(p-1)}{n-1}w_1^{p-2}w_{11}\Delta_pw.
\end{align*}
Since
$$\nabla F=pw_1^{p-1}w_{1j}e_j,$$
we have
\begin{align}\label{awff}
&\hin{A_w\left(\nabla F\right)}{\nabla F}=p^2(p-1)w_1^{2p-2}\left((p-1)w_{11}^2+\sum_{j=2}^nw_{1j}^2\right),
\end{align}
and
\begin{align}\label{awfw}
\hin{A_{w}\left(\nabla F\right)}{\mathbf{W}}=p(p-1)w_{1}^{2p-2}w_{11}.
\end{align}
It follows
\begin{align*}
\hin{\nabla_{\nabla_{e_i}\mathbf{W}}\mathbf{W}}{e_i}\geq&\dfrac{1}{n-1}\abs{\Delta_pw}^2+\dfrac{n(p-1)}{(n-1)p^2}F^{-2/p}\hin{A_w\left(\nabla F\right)}{\nabla F}-\dfrac{2}{(n-1)p}F^{-1}\hin{A_{w}\left(\nabla F\right)}{\mathbf{W}}\Delta_pw.
\end{align*}

Notice that
\begin{align*}
A_{w}\left(\nabla w\right)=(p-1)\nabla w.
\end{align*}
Inserting the above Kato inequality \eqref{eq:kato-w} into the Bochner formula \eqref{eq:bochner-w}, we obtain
\begin{align*}
\dfrac{1}{p}\mathcal{L}_{p,w}F\geq&\dfrac{1}{n-1}\abs{\Delta_pw}^2-\dfrac{2}{(n-1)p}F^{-1}\hin{A_w\left(\nabla F\right)}{\mathbf{W}}\Delta_pw+\dfrac{n(p-1)}{(n-1)p^2}F^{-2/p}\hin{A_w\left(\nabla F\right)}{\nabla F}\\
&+(1-p)\hin{\nabla F}{\mathbf{W}}-\left((1-p)u^{1-p}f(u)+u^{2-p}f'(u)\right)F -(n-1)\kappa F^{2-2/p}\\
=&\dfrac{1}{n-1}\left((1-p)F-u^{1-p}f(u)\right)^2-\dfrac{2}{(n-1)p}F^{-1}\hin{A_w\left(\nabla F\right)}{\mathbf{W}}\left((1-p)F-u^{1-p}f(u)\right)\\
&+\dfrac{n(p-1)}{(n-1)p^2}F^{-2/p}\hin{A_w\left(\nabla F\right)}{\nabla F}-\hin{A_w\left(\nabla F\right)}{\mathbf{W}}-\left((1-p)u^{1-p}f(u)+u^{2-p}f'(u)\right)F\\
&-(n-1)\kappa F^{2-2/p}\\
=&\dfrac{(p-1)^2F^2}{n-1}+\dfrac{u^{2-2p}f(u)^2}{n-1}+\dfrac{2}{(n-1)p}u^{1-p}f(u)F^{-1}\hin{A_w\left(\nabla F\right)}{\mathbf{W}}+\dfrac{n(p-1)}{(n-1)p^2}F^{-2/p}\hin{A_w\left(\nabla F\right)}{\nabla F}\\
&-\dfrac{(n-3)p+2}{(n-1)p}\hin{A_w\left(\nabla F\right)}{\mathbf{W}}+\left(\dfrac{(n+1)(p-1)}{n-1}u^{1-p}f(u)-u^{2-p}f'(u)\right)F-(n-1)\kappa F^{2-2/p}.
\end{align*}
By \eqref{awff} and \eqref{awfw}, we have
\begin{align*}
|\langle A_w(\nabla F), \mathbf{W}\rangle|\leq \sqrt{p-1}F^{1-1/p}\sqrt{\langle A_w(\nabla F), \nabla F\rangle}.
\end{align*}
Hence, it follows from the above two inequalities that
\begin{align}
\label{eq:4.6}
\begin{split}
\dfrac{1}{p}\mathcal{L}_{p,w}F \geq&\dfrac{(p-1)^2F^2}{n-1}+\dfrac{u^{2-2p}f(u)^2}{n-1}-\dfrac{2\sqrt{p-1}}{(n-1)p}u^{1-p}\abs{f(u)}F^{-1/p}\sqrt{\hin{A_w\left(\nabla F\right)}{\nabla F}}\\
&+\dfrac{n(p-1)}{(n-1)p^2}F^{-2/p}\hin{A_w\left(\nabla F\right)}{\nabla F}-\dfrac{\abs{(n-3)p+2}\sqrt{p-1}}{(n-1)p}F^{1-1/p}\sqrt{\hin{A_w\left(\nabla F\right)}{\nabla F}}
\\
&+\left(\dfrac{(n+1)(p-1)}{n-1}u^{1-p}f(u)-u^{2-p}f'(u)\right)F -(n-1)\kappa F^{2-2/p}.
\end{split}
\end{align}
We note that for any $\delta\in[0,1)$ there holds
$$
 \delta \dfrac{(p-1)^2F^2}{n-1}+ \delta \dfrac{u^{2-2p}f(u)^2}{n-1}-\dfrac{2\delta(p-1)}{n-1}u^{1-p}\abs{f(u)}F\geq 0.
 $$
Now, inserting the above inequality into \eqref{eq:4.6} we get
\begin{align*}
\dfrac{1}{p}\mathcal{L}_{p,w}F \geq&\left(1-\delta\right)\dfrac{(p-1)^2F^2}{n-1}+\left(1-\delta\right)\dfrac{u^{2-2p}f(u)^2}{n-1}+\dfrac{2\delta(p-1)}{n-1}u^{1-p}\abs{f(u)}F\\
&-\dfrac{2\sqrt{p-1}}{(n-1)p}u^{1-p}\abs{f(u)}F^{-1/p}\sqrt{\hin{A_w\left(\nabla F\right)}{\nabla F}}+\dfrac{n(p-1)}{(n-1)p^2}F^{-2/p}\hin{A_w\left(\nabla F\right)}{\nabla F}\\
&-\dfrac{\abs{(n-3)p+2}\sqrt{p-1}}{(n-1)p}F^{1-1/p}\sqrt{\hin{A_w\left(\nabla F\right)}{\nabla F}}+\left(\dfrac{(n+1)(p-1)}{n-1}u^{1-p}f(u)-u^{2-p}f'(u)\right)F\\
&-(n-1)\kappa F^{2-2/p}\\
\geq&\left(1-\delta\right)\dfrac{(p-1)^2F^2}{n-1}+\dfrac{(n(1-\delta)-1)(p-1)}{(1-\delta)(n-1)p^2}F^{-2/p}\hin{A_w\left(\nabla F\right)}{\nabla F}\\
&-\dfrac{\abs{(n-3)p+2}\sqrt{p-1}}{(n-1)p}F^{1-1/p}\sqrt{\hin{A_w\left(\nabla F\right)}{\nabla F}}\\
&+\left(\dfrac{p-1}{n-1}\left((n+1)f(u)+2\delta\abs{f(u)}\right)-uf'(u)\right)u^{1-p}F -(n-1)\kappa F^{2-2/p}
\end{align*}
for every $\delta\in[0,1)$. Hence, noting that there hold for $\delta=\delta_0^+=\max\set{0,\delta_0}\in[0,1)$
\begin{align*}
\dfrac{(n(1-\delta_0^+)-1)(p-1)}{(1-\delta_0^+)(n-1)p^2}=\dfrac{((n-1)(1-\delta_0^+)-\delta_0^+)(p-1)}{(1-\delta_0^+)(n-1)p^2}
\geq -\dfrac{\delta_0^+}{(1-\delta_0^+)p},
\end{align*}
and
$$ \dfrac{\abs{(n-3)p+2}}{(n-1)p}\leq 1, \quad\mbox{for}\,\, p>1,$$
we conclude
\begin{align}\label{eq:F-1}
\dfrac{1}{p}\mathcal{L}_{p,w}F\geq&\left(1-\delta_0^+\right)\dfrac{(p-1)^2F^2}{n-1}-\dfrac{\delta_0^+}{(1-\delta_0^+)p}F^{-2/p}\hin{A_w\left(\nabla F\right)}{\nabla F}-(p-1)F^{1-1/p}\abs{\nabla F}-(n-1)\kappa F^{2-2/p}.
\end{align}

For every real number $\lambda>0$, from \eqref{eq:F-1} we have
\begin{align*}
\dfrac{1}{\lambda}\mathcal{L}_{p,w}F^{\lambda}=&F^{\lambda-1}\mathcal{L}_{p,w}F+\left(\lambda-1\right)F^{\lambda-1-2/p}\hin{A_w\left(\nabla F\right)}{\nabla F}\\
\geq&pF^{\lambda-1}\left(\left(1-\delta_0^+\right)\dfrac{(p-1)^2F^2}{n-1}-\dfrac{\delta_0^+}{(1-\delta_0^+)p}F^{-2/p}\hin{A_w\left(\nabla F\right)}{\nabla F}-(p-1)F^{1-1/p}\abs{\nabla F}\right.\\
&\left.-(n-1)\kappa F^{2-2/p}\right)+\left(\lambda-1\right)F^{\lambda-1-2/p}\hin{A_w\left(\nabla F\right)}{\nabla F}\\
=&\left(1-\delta_0^+\right)\dfrac{(p-1)^2p}{n-1}F^{\lambda+1}+\left(\lambda-\dfrac{1}{1-\delta_0^+}\right)F^{\lambda-1-2/p}\hin{A_w\left(\nabla F\right)}{\nabla F}\\
&-p(p-1)F^{\lambda-1/p}\abs{\nabla F}-p(n-1)\kappa F^{\lambda+1-2/p}.
\end{align*}
Therefore, for every $\lambda\geq \frac{1}{1-\delta_0^+}$, we get
\begin{align*}
\dfrac{1}{\lambda}\mathcal{L}_{p,w}F^{\lambda}
\geq&\left(1-\delta_0^+\right)\dfrac{(p-1)^2p}{n-1}F^{\lambda+1}-p(n-1)\kappa F^{\lambda+1-2/p}-p(p-1)F^{\lambda-1/p}\abs{\nabla F}.
\end{align*}
One picks $\lambda_0=\frac{p}{1-\delta_0^+}$ to  complete the proof.
\end{proof}

Secondly, we state the following local $L^q$-boundedness.
\begin{lem}\label{lem:lq}
Let $M^n$ be a complete and $n$-dimensional Riemannian manifold with the Ricci curvature bounded from below by a nonpositive constant $-(n-1)\kappa$. Assume for some constant $\delta_0<1$,
\begin{align*}
\dfrac{p-1}{n-1}\left((n+1)f(u)+2\delta_0\abs{f(u)}\right)-uf'(u)\geq0.
\end{align*}
If $p>1$ and $u$ is a positive and weak solution to \eqref{eq:quasi-h}, then there exists a positive constant $q_0$ depending only on $n,p$ and $\delta_0^+$ such that for all $q\geq q_0$
\begin{align}\label{eq:lq}
\norm{\nabla\ln u}_{L^{q}\left(B_{R/2}\right)}\leq \dfrac{C_{n,p}}{\sqrt{1-\delta_0}}\left(\dfrac{q}{R}+\sqrt{\kappa}\right)V_{R}^{\frac{1}{q}},\quad\forall R>0.
\end{align}
Here $C_{n,p}$ is a positive constant which is depending only on $n$ and $p$.
\end{lem}

\begin{proof}
Let $\lambda_0=\frac{p}{1-\delta_0^+}$ be the constant in \autoref{lem:moser}. Consider the function
\begin{align*}
H=\abs{\nabla\ln u}.
\end{align*}
It is easy to see that it follows from \eqref{eq:moser}
\begin{align*}
\dfrac{1}{\lambda_0}\mathcal{L}_{p,w}H^{\lambda_0}\geq\dfrac{(p-1)^2\left(1-\delta_0^+\right)}{n-1}H^{\lambda_0+p}-(n-1)\kappa H^{\lambda_0+p-2}-p(p-1)H^{\lambda_0+p-2}\abs{\nabla H},
\end{align*}
provided $H\neq0$. As a consequence,  for any nonnegative cutoff function  $\eta$ and positive number $t$, we consider the test function $\eta^2H^{t}$ and obtain
\begin{align}\label{eq:local-lq0}
&\dfrac{(p-1)^2\left(1-\delta_0^+\right)}{n-1}\int_{M}H^{\lambda_0+p+t}\eta^2-(n-1)\kappa \int_MH^{\lambda_0+p+t-2}\eta^2\nonumber\\
\leq &p(p-1)\int_MH^{\lambda_0+p+t-2}\abs{\nabla H}\eta^2+\dfrac{1}{\lambda_0}\int_{M}\eta^2H^{t}\mathcal{L}_{p,w}H^{\lambda_0}.
\end{align}
Here we have used the following well known fact about the regularity of $u$ (see \cite{T1984} or \cite{Antonini2023}):
\begin{align*}
u\in C_{loc}^{\infty}\left(M\setminus Z\right)\cap W^{2,2}_{loc}\left(M\setminus Z\right)\cap C^{1,\beta}_{loc}\left(M\right),
\end{align*}
where $Z=\set{x\in M:\nabla u = 0}$ is the critical point set of $u$ on $M$, and for each compact domain $\Omega\Subset M$
\begin{align*}
\abs{\nabla u}^{p-2}\nabla^2u\in L^2\left(\Omega\setminus \Omega_{cr}\right)
\end{align*}
where $\Omega_{cr}=\set{x\in\Omega:\nabla u = 0}$ is the set of critical points of $u$ on $\Omega$.

We need to derive some estimates for the terms on the right hand side of \eqref{eq:local-lq0}. Notice that the eigenvalues of $A_{u}$ are bounded from below by $\min\set{1,p-1}\geq\frac{p-1}{p}$ and above by $\max\set{1,p-1}\leq p$. On the one hand,
\begin{align*}
p(p-1)\int_MH^{\lambda_0+p+t-2}\abs{\nabla H}\eta^2\leq&\dfrac{2p\sqrt{p(p-1)}}{p+\lambda_0+t-2}\int_MH^{\frac{\lambda_0+p+t}{2}}\sqrt{\hin{A_w\left(\nabla H^{\frac{p+\lambda_0+t-2}{2}}\right)}{\nabla H^{\frac{p+\lambda_0+t-2}{2}}}}\eta^2\\
\leq&\dfrac{2p\sqrt{p(p-1)}}{p+\lambda_0+t-2}\left(\int_MH^{\lambda_0+p+t}\eta^2\right)^{1/2}\left(\int_M\hin{A_w\left(\nabla H^{\frac{p+\lambda_0+t-2}{2}}\right)}{\nabla H^{\frac{p+\lambda_0+t-2}{2}}}\eta^2\right)^{1/2}\\
\leq&\dfrac{1}{p+\lambda_0+t-2}\int_M\hin{A_w\left(\nabla H^{\frac{p+\lambda_0+t-2}{2}}\right)}{\nabla H^{\frac{p+\lambda_0+t-2}{2}}}\eta^2\\
&+\dfrac{p^3(p-1)}{p+\lambda_0+t-2}\int_MH^{\lambda_0+p+t}\eta^2.
\end{align*}
On the other hand,
\begin{align*}
\dfrac{1}{\lambda_0}\int_{M}\eta^2H^{t}\mathcal{L}_{p,w}H^{\lambda_0}=&-\int_M H^{p+\lambda_0-3}\hin{A_{w}\left(\nabla H\right)}{2\eta H^t\nabla\eta+t\eta^2H^{t-1}\nabla H}\\
=&-t\int_M H^{p+\lambda_0-1}H_a^{t-1}\hin{A_{w}\left(\nabla H\right)}{\nabla H}\eta^2-2\int_M H^{p+\lambda_0+t-3}\hin{A_{w}\left(\nabla H\right)}{\nabla\eta}\eta\\
=&-\dfrac{4t}{\left(p+\lambda_0+t-2\right)^2}\int_M \hin{A_{w}\left(\nabla H^{\frac{p+\lambda_0+t-2}{2}}\right)}{\nabla H^{\frac{p+\lambda_0+t-2}{2}}}\eta^2\\
&-\dfrac{4}{p+\lambda_0+t-2}\int_M\hin{A_w\left(H^{\frac{p+\lambda_0+t-2}{2}}\right)}{\nabla\eta}H^{\frac{p+\lambda_0+t-2}{2}}\eta.
\end{align*}
Integration by parts, we get
\begin{align*}
\dfrac{1}{\lambda_0}\int_{M}\eta^2H^{t}\mathcal{L}_{p,w}H^{\lambda_0}=&-\dfrac{4t}{\left(p+\lambda_0+t-2\right)^2}\int_M \hin{A_{w}\left(\nabla H^{\frac{p+\lambda_0+t-2}{2}}\right)}{\nabla H^{\frac{p+\lambda_0+t-2}{2}}}\eta^2\\
&+\dfrac{2}{p+\lambda_0+t-2}\int_M \hin{A_{w}\left(\nabla H^{\frac{p+\lambda_0+t-2}{2}}\right)}{\nabla H^{\frac{p+\lambda_0+t-2}{2}}}\eta^2\\
&+\dfrac{2}{p+\lambda_0+t-2}\int_M \hin{A_{w}\left(\nabla \eta\right)}{\nabla \eta}H^{p+\lambda_0+t-2}\\
&-\dfrac{2}{p+\lambda_0+t-2}\int_M \hin{A_{w}\left(\nabla \left(\eta H^{\frac{p+\lambda_0+t-2}{2}}\right)\right)}{\nabla \left(\eta H^{\frac{p+\lambda_0+t-2}{2}}\right)}\\
=&-\dfrac{2\left(t+2-p-\lambda_0\right)}{\left(p+\lambda_0+t-2\right)^2}\int_M \hin{A_{w}\left(\nabla H^{\frac{p+\lambda_0+t-2}{2}}\right)}{\nabla H^{\frac{p+\lambda_0+t-2}{2}}}\eta^2\\
&+\dfrac{2}{p+\lambda_0+t-2}\int_M \hin{A_{w}\left(\nabla \eta\right)}{\nabla \eta}H^{p+\lambda_0+t-2}\\
&-\dfrac{2}{p+\lambda_0+t-2}\int_M \hin{A_{w}\left(\nabla \left(\eta H^{\frac{p+\lambda_0+t-2}{2}}\right)\right)}{\nabla \left(\eta H^{\frac{p+\lambda_0+t-2}{2}}\right)}.
\end{align*}
Thus
\begin{align*}
&\dfrac{1}{\lambda_0}\int_{M}\eta^2H^{t}\mathcal{L}_{p,w}H^{\lambda_0}+p(p-1)\int_MH^{\lambda_0+p+t-2}\abs{\nabla H}\eta^2\\
\leq&-\dfrac{2\left(t+2-p-\lambda_0\right)}{\left(p+\lambda_0+t-2\right)^2}\int_M \hin{A_{w}\left(\nabla H^{\frac{p+\lambda_0+t-2}{2}}\right)}{\nabla H^{\frac{p+\lambda_0+t-2}{2}}}\eta^2+\dfrac{2}{p+\lambda_0+t-2}\int_M \hin{A_{w}\left(\nabla \eta\right)}{\nabla \eta}H^{p+\lambda_0+t-2}\\
&-\dfrac{2}{p+\lambda_0+t-2}\int_M \hin{A_{w}\left(\nabla \left(\eta H^{\frac{p+\lambda_0+t-2}{2}}\right)\right)}{\nabla \left(\eta H^{\frac{p+\lambda_0+t-2}{2}}\right)}\\
&+\dfrac{1}{p+\lambda_0+t-2}\int_M\hin{A_w\left(\nabla H^{\frac{p+\lambda_0+t-2}{2}}\right)}{\nabla H^{\frac{p+\lambda_0+t-2}{2}}}\eta^2+\dfrac{p^3(p-1)}{p+\lambda_0+t-2}\int_MH^{\frac{p+\lambda_0+t-2}{2}}\eta^2\\
=&-\dfrac{t-3\left(p+\lambda_0-2\right)}{\left(p+\lambda_0+t-2\right)^2}\int_M \hin{A_{w}\left(\nabla H^{\frac{p+\lambda_0+t-2}{2}}\right)}{\nabla H^{\frac{p+\lambda_0+t-2}{2}}}\eta^2+\dfrac{2}{p+\lambda_0+t-2}\int_M \hin{A_{w}\left(\nabla \eta\right)}{\nabla \eta}H^{p+\lambda_0+t-2}\\
&-\dfrac{2}{p+\lambda_0+t-2}\int_M \hin{A_{w}\left(\nabla \left(\eta H^{\frac{p+\lambda_0+t-2}{2}}\right)\right)}{\nabla \left(\eta H^{\frac{p+\lambda_0+t-2}{2}}\right)}+\dfrac{p^3(p-1)}{p+\lambda_0+t-2}\int_M H^{p+\lambda_0+t}\eta^2.
\end{align*}
Taking
\begin{align*}
t_0=3(p+\lambda_0-2)+\dfrac{2(p-1)p^3}{\epsilon_0}\quad\mbox{and}\quad \epsilon_0=\dfrac{(p-1)^2\left(1-\delta_0^+\right)}{n-1},
\end{align*}
we get for all $t\geq t_0$
\begin{align}\label{eq:local-t_0}
\begin{split}
&\dfrac{1}{\lambda_0}\int_{M}\eta^2H^{t}\mathcal{L}_{p,w}H^{\lambda_0}+p(p-1)\int_MH^{\lambda_0+p+t-2}\abs{\nabla H}\eta^2\\
\leq&\dfrac{2}{p+\lambda_0+t-2}\int_M \hin{A_{w}\left(\nabla \eta\right)}{\nabla \eta}H^{p+\lambda_0+t-2}-\dfrac{2}{p+\lambda_0+t-2}\int_M \hin{A_{w}\left(\nabla \left(\eta H^{\frac{p+\lambda_0+t-2}{2}}\right)\right)}{\nabla \left(\eta H^{\frac{p+\lambda_0+t-2}{2}}\right)}\\
&+\dfrac{\epsilon_0}{2}\int_M H^{p+\lambda_0+t}\eta^2.
\end{split}
\end{align}
Combing \eqref{eq:local-lq0} with \eqref{eq:local-t_0}, we obtain
\begin{align*}
\dfrac{\epsilon_0}{2}\int_{M}H^{\lambda_0+p+t}\eta^2\leq& \dfrac{2}{p+\lambda_0+t-2}\int_M \hin{A_{w}\left(\nabla \eta\right)}{\nabla \eta}H^{p+\lambda_0+t-2}\\
&-\dfrac{2}{p+\lambda_0+t-2}\int_M \hin{A_{w}\left(\nabla \left(\eta H^{\frac{p+\lambda_0+t-2}{2}}\right)\right)}{\nabla \left(\eta H^{\frac{p+\lambda_0+t-2}{2}}\right)}.
\end{align*}
Moreover, we get
\begin{align*}
&\dfrac{\epsilon_0}{2}\int_{M}H^{\lambda_0+p+t}\eta^2+\dfrac{2(p-1)}{p(p+\lambda_0+t-2)}\int_M \abs{\nabla \left(\eta H^{\frac{p+\lambda_0+t-2}{2}}\right)}^2\\
\leq& \dfrac{2p}{p+\lambda_0+t-2}\int_M \abs{\nabla \eta}^2H^{p+\lambda_0+t-2}+(n-1)\kappa \int_MH^{\lambda_0+p+t-2}\eta^2.
\end{align*}
Hence, for all
$$t\geq t_1\coloneqq \frac{p+\lambda_0+t_0-2}{2},$$
we have the following integer inequality
\begin{align}\label{eq:moser-1}
\dfrac{\epsilon_0}{2}\int_{M}H^{2t+2}\eta^2+\dfrac{p-1}{pt}\int_M \abs{\nabla \left(\eta H^{t}\right)}^2\leq \dfrac{p}{t}\int_M \abs{\nabla \eta}^2H^{2t}+(n-1)\kappa \int_MH^{2t}\eta^2.
\end{align}
Now, assume that $\eta$ is supported in $B_R$ and denote by $V_R=\mathrm{Vol}\left(B_R\right)$. Using the Sobolev inequality \eqref{eq:sobolev}, it follows from \eqref{eq:moser-1},  we have
\begin{align}\label{eq:moser-2}
\dfrac{\epsilon_0}{2}\fint_{B_{R}}H^{2t+2}\eta^2+\dfrac{e^{-C_{n,p}\left(1+\sqrt{\kappa}R\right)}}{tR^2}\left(\fint_{B_R} \left(\eta^2 H^{2t}\right)^{\chi}\right)^{1/\chi}\leq\dfrac{p}{t}\fint_{B_R} H^{2t}\abs{\nabla\eta}^2+(n-1)\kappa \fint_{B_{R}}H^{2t}\eta^2,\quad\forall t\geq t_1.
\end{align}
In particular, replacing $\eta$ by $\eta^{t+1}$, and using H\"older's inequality with the respective exponent pairs
\begin{align*}
\left(\dfrac{t+1}{t},\, t+1\right),
\end{align*}
 we obtain from \eqref{eq:moser-2}
\begin{align*}
\dfrac{\epsilon_0}{2}\int_{B_{R}}H^{2t+2}\eta^{2t+2}\leq&\dfrac{p(t+1)^2}{t}\int_{B_R} H^{2t}\abs{\nabla\eta}^2\eta^{2t}+(n-1)\kappa \int_{B_{R}}H^{2t}\eta^{2t+2}\\
\leq&\dfrac{2p(t+1)^2}{t}\left(\int_{B_{R}}H^{2t+2}\eta^{2t+2}\right)^{\frac{t}{t+1}}\left(\int_{B_{R}}\abs{\nabla\eta}^{2t+2}\right)^{\frac{1}{t+1}}+2(n-1)\kappa \left(\int_{B_{R}}H^{2t+2}\eta^{2t+2}\right)^{\frac{t}{t+1}}\left(\int_{B_{R}}\eta^{2t+2}\right)^{\frac{1}{t+1}},
\end{align*}
which implies
\begin{align}\label{eq:local-lq1}
\dfrac{\epsilon_0}{2}\left(\int_{B_{R}}H^{2t+2}\eta^{2t+2}\right)^{\frac{1}{1+t}}\leq&\dfrac{2p(t+1)^2}{t}\left(\int_{B_{R}}\abs{\nabla\eta}^{2t+2}\right)^{\frac{1}{t+1}}+2(n-1)\kappa \left(\int_{B_{R}}\eta^{2t+2}\right)^{\frac{t}{t+1}}.
\end{align}

We are now in a position to state the local $L^q$-estimate \eqref{eq:lq} based on the above inequality \eqref{eq:local-lq1}. To see this, choose $\eta\in C_0^{\infty}\left(B_R\right)$ such that
\begin{align*}
0\leq\eta\leq1,\quad\eta\vert_{B_{R/2}}=1,\quad\abs{\nabla\eta}\leq 4/R
\end{align*}
It follows from \eqref{eq:local-lq1}
\begin{align*}
\left(\int_{B_{R/2}}H^{2t+2}\right)^{\frac{1}{1+t}}\leq&\left(\dfrac{64p(t+1)^2}{\epsilon_0tR^2}+\dfrac{4(n-1)\kappa}{\epsilon_0}\right)V_R^{\frac{1}{1+t}}.
\end{align*}
Notice that $\epsilon_0=\frac{(p-1)^2\left(1-\delta_0^+\right)}{n-1}$ and
\begin{align*}
t_1=2\left(p+\dfrac{p}{1-\delta_0^+}-2\right)+\dfrac{(n-1)p^3}{(p-1)\left(1-\delta_0^+\right)}> np-1>(n-1)p.
\end{align*}
Hence for every $t\geq t_1$
\begin{align*}
\dfrac{64p(t+1)^2}{\epsilon_0tR^2}+\dfrac{4(n-1)\kappa}{\epsilon_0}
= &\dfrac{64(n-1)p(t+1)^2}{(p-1)^2\left(1-\delta_0^+\right)tR^2}+\dfrac{4(n-1)^2\kappa}{(p-1)^2\left(1-\delta_0^+\right)}\\
\leq &\dfrac{64(t+1)^2}{(p-1)^2\left(1-\delta_0^+\right)R^2}+\dfrac{4(n-1)^2\kappa}{(p-1)^2\left(1-\delta_0^+\right)}.
\end{align*}
We conclude that there exists a constant $q_0=2t_1+2$ such that for all $q\geq q_0$,
\begin{align*}
\norm{H}_{L^q\left(B_{R/2}\right)}\leq\dfrac{2(n+1)}{p-1}\dfrac{1}{\sqrt{1-\delta_0^+}}V_R^{1/q}\left(\dfrac{q}{R}+\sqrt{\kappa}\right).
\end{align*}
We complete the proof.
\end{proof}

Now we can use the Moser iteration to obtain the local $L^{\infty}$-boundedness.

\begin{proof}[Proof of \autoref{thm:boundedness}]
We begin with the estimate \eqref{eq:moser-2} and discard the first term of the left hand side of \eqref{eq:moser-2} to obtain: for $H=\abs{\nabla u}$ and $t\geq t_1$, the following
\begin{align}\label{eq:moser-11}
\left(\int_{B_R} \left(\eta^2 H^{2t}\right)^{\chi}\right)^{1/\chi}\leq pe^{C_{n,p}\left(1+\sqrt{\kappa}R\right)}R^2V_{R}^{-\frac{\chi-1}{\chi}}\int_{B_R} H^{2t}\abs{\nabla\eta}^2+(n-1)e^{C_{n,p}\left(1+\sqrt{\kappa}R\right)}\kappa R^2tV_{R}^{-\frac{\chi-1}{\chi}}\int_{B_{R}}H^{2t}\eta^2
\end{align}
holds for any cutoff function $\eta$ supported in $B_R$. Fixed $q_0$ as in \autoref{lem:lq} and set
\begin{align*}
\tilde q_0=q_0+\sqrt{\kappa}R.
\end{align*}
One can check that for some uniform constant $\tilde C_{n,p}$,
\begin{align*}
q_0\leq\dfrac{\tilde C_{n,p}}{1-\delta_0^+}.
\end{align*}

To apply the Moser iteration, we choose the sequences of $R_i$ and $q_i$ such that
\begin{align*}
R_i=\dfrac{R}{4}\left(1+\dfrac{1}{2^i}\right),\quad q_i=\tilde q_0\chi^{i},\quad i=0,1,2,\dots.
\end{align*}
For each $i$, we can choose a cutoff function  $\eta_i\in C_0^{\infty}\left(B_{R_i}\right)$ satisfying
\begin{align*}
0\leq\eta_i\leq 1,\quad \eta_i\vert_{B_{R_{i+1}}}=1,\quad \abs{\nabla\eta_i}\leq 2^{i+4}/R.
\end{align*}
Using the Sobolev inequality \eqref{eq:sobolev}, we can see that the following follows from \eqref{eq:moser-11}
\begin{align*}
\left(\int_{B_{R}} \left(\eta_i^2 H^{q_i}\right)^{\chi}\right)^{1/\chi}\leq pe^{C_{n,p}\left(1+\sqrt{\kappa}R\right)}R^2V_{R}^{-\frac{\chi-1}{\chi}}\int_{B_R} H^{q_i}\abs{\nabla\eta_i}^2+(n-1)e^{C_{n,p}\left(1+\sqrt{\kappa}R\right)}\kappa R^2\dfrac{q_i}{2}V_{R}^{-\frac{\chi-1}{\chi}}\int_{B_{R}}H^{q_i}\eta_i^2
\end{align*}
which gives
\begin{align*}
\left(\int_{B_{R_{i+1}}}  H^{q_{i+1}}\right)^{1/\chi}\leq \dfrac{e^{L_{n,p}C_{n,p}\left(1+\sqrt{\kappa}R\right)}}{1-\delta_0^+}V_{R}^{-\frac{\chi-1}{\chi}}\left(4+\chi\right)^i\int_{B_{R_i}} H^{q_i},\quad\forall i=0,1,2,\dotsc.
\end{align*}
Here $L_{n,p}$ is a uniform constant depending only on $n$ and $p$. Thus
\begin{align*}
\norm{H}_{L^{q_{i+1}}\left(B_{R_{i+1}}\right)}\leq \left(\dfrac{e^{L_{n,p}C_{n,p}\left(1+\sqrt{\kappa}R\right)}}{1-\delta_0^+}V_{R}^{-\frac{\chi-1}{\chi}}\left(4+\chi\right)^i\right)^{1/q_i} \norm{H}_{L^{q_{i}}\left(B_{R_{i}}\right)},~~\forall i=0,1,2,\dots.
\end{align*}

Since
\begin{align*}
\Pi_{i=0}^{\infty}\left(\dfrac{e^{L_{n,p}C_{n,p}\left(1+\sqrt{\kappa}R\right)}}{1-\delta_0^+}V_{R}^{-\frac{\chi-1}{\chi}}\left(4+\chi\right)^i\right)^{1/q_i} =&\left(\dfrac{e^{L_{n,p}C_{n,p}\left(1+\sqrt{\kappa}R\right)}}{1-\delta_0^+}V_{R}^{-\frac{\chi-1}{\chi}}\right)^{\sum_{i=0}^{\infty}1/q_i}\left(4+\chi\right)^{\sum_{i=0}^{\infty}i/q_i}\\
=&\left(\dfrac{e^{L_{n,p}C_{n,p}\left(1+\sqrt{\kappa}R\right)}}{1-\delta_0^+}\right)^{\frac{\chi}{\tilde q_0(\chi-1)}}V_R^{-1/\tilde q_0}\left(4+\chi\right)^{\frac{\chi}{\tilde q_0(\chi-1)^2}},
\end{align*}
we conclude that
\begin{align*}
\norm{H}_{L^{\infty}\left(B_{R/4}\right)}\leq& \left(\dfrac{e^{L_{n,p}C_{n,p}\left(1+\sqrt{\kappa}R\right)}}{1-\delta_0^+}\right)^{\frac{\chi}{\tilde q_0(\chi-1)}}V_R^{-1/\tilde q_0}\left(4+\chi\right)^{\frac{\chi}{\tilde q_0(\chi-1)^2}}\norm{H}_{L^{\tilde q_0}\left(B_{R/2}\right)}.
\end{align*}
In view of the local $L^q$-estimate \autoref{lem:lq}, we conclude
\begin{align*}
\norm{H}_{L^{\infty}\left(B_{R/4}\right)}\leq& \left(\dfrac{e^{L_{n,p}C_{n,p}\left(1+\sqrt{\kappa}R\right)}}{1-\delta_0^+}\right)^{\frac{\chi}{\tilde q_0(\chi-1)}}V_R^{-1/\tilde q_0}\left(4+\chi\right)^{\frac{\chi}{\tilde q_0(\chi-1)^2}}C_{n,p,\delta_0^+}\left(\dfrac{\tilde q_0}{R}+\sqrt{\kappa}\right)V_{R}^{\frac{1}{\tilde q_0}}\\
\leq& C_{n,p,\delta_0^+}\left(\dfrac{1}{R}+\sqrt{\kappa}\right).
\end{align*}
Thus, we complete the proof of this lemma.
\end{proof}

According to the local gradient estimate established in \autoref{thm:boundedness}, we can take an argument analogous to the approach adopted in \cite{Sung2014} to obtain the global gradient estimate. That is, we give a proof of \autoref{thm:global-gradient-estimate}.

\begin{proof}[Proof of \autoref{thm:global-gradient-estimate}]
According to \autoref{lem:moser}, for $\lambda_0=\frac{p}{1-\delta_0^+}$, there holds
\begin{align}\label{eq:global-1}
\dfrac{1}{\lambda_0}\mathcal{L}_{p,w}H^{\lambda_0}\geq&\dfrac{(p-1)^2\left(1-\delta_0^+\right)}{n-1}H^{\lambda_0+p}-(n-1)\kappa H^{\lambda_0+p-2}- p(p-1)H^{\lambda_0+p-2}\abs{\nabla H}.
\end{align}
Without loss of generality, assume $M$ is noncompact.

According to \autoref{thm:boundedness}, we know that the function $H\coloneqq\abs{\nabla u}$ is uniformly bounded on $M$. Denote
\begin{align*}
\Lambda_0\coloneqq\dfrac{n-1}{p-1}\sqrt{\dfrac{\kappa}{1-\delta_0^+}}.
\end{align*}
Fixed $a\in\left(\Lambda_0,\, \Lambda\right]$ where
\begin{align*}
\Lambda_0<\Lambda\coloneqq\sup_{M}H<\infty.
\end{align*}

Notice that for every $\tau>0$, the inequality \eqref{eq:global-1} implies
\begin{align*}
\mathcal{L}_{p,w}H^{\tau\lambda_0}=&\tau H^{(\tau-1)\lambda_0}\mathcal{L}_{p,w}H^{\tau\lambda_0}+\tau(\tau-1)H^{(\tau-2)\lambda_0}H^{p-2}\hin{A_w\left(\nabla H^{\lambda_0}\right)}{\nabla H^{\lambda_0}}\\
\geq&\tau H^{(\tau-1)\lambda_0}\lambda_0\left(\dfrac{(p-1)^2\left(1-\delta_0^+\right)}{n-1}H^{\lambda_0+p}-(n-1)\kappa H^{\lambda_0+p-2}- p(p-1)H^{\lambda_0+p-2}\abs{\nabla H}\right)\\
&+\dfrac{\tau(\tau-1)(p-1)}{p}H^{(\tau-2)\lambda_0}H^{p-2}\abs{\nabla H^{\lambda_0}}^2\\
\geq&\tau H^{\tau\lambda_0}\lambda_0\left(\left(\dfrac{(p-1)^2\left(1-\delta_0^+\right)}{n-1}-\dfrac{p^3(p-1)}{4(\tau-1)\lambda_0}\right)
H^{p}-(n-1)\kappa H^{p-2}\right)\\
=&\tau H^{\tau\lambda_0}\lambda_0\left(\dfrac{(p-1)^2\left(1-\delta_0^+\right)}{n-1}\left(1-\dfrac{(n-1)p^2}{4(p-1)(\tau-1)}\right)
H^{p}-(n-1)\kappa H^{p-2}\right).
\end{align*}
If $H\geq a$ and
\begin{align*}
1-\dfrac{(n-1)p^2}{4(p-1)(\tau-1)}-\dfrac{\Lambda_0^2}{a^2}>0,
\end{align*}
then
\begin{align*}
\mathcal{L}_{p,w}H^{\tau\lambda_0}\geq\,\tau \lambda_0(n-1)\kappa\left(\dfrac{a^2}{\Lambda_0^2}\left(1-\dfrac{(n-1)p^2}{4(p-1)(\tau-1)}\right)-1\right)H^{p-2}H^{\tau\lambda_0}.
\end{align*}
Hence, by the choice of $a$, there exists a positive number $\tau_0$ such that the following inequality holds in $\set{H\geq a}$ for all $\tau\geq\tau_0$
\begin{align*}
\mathcal{L}_{p,w}H^{\tau\lambda_0}\geq\,&\tau C_{n,p,\kappa,\delta_0^+,a} H^{p-2}H^{\tau\lambda_0}.
\end{align*}
Here
\begin{align*}
C_{n,p,\kappa,\delta_0^+,a}=\lambda_0(n-1)\kappa\left(\dfrac{a^2}{\Lambda_0^2}\left(1-\dfrac{(n-1)p^2}{4(p-1)(\tau_0-1)}\right)-1\right)>0.
\end{align*}

Denote
$$\psi_a=\left(H^{\tau_0\lambda_0}-a^{\tau_0\lambda_0}\right)^+.$$
For every $t\geq1$ and nonnegative cutoff function  $\eta\in C_0^{\infty}\left(B_{r}\right)$, we have
\begin{align*}
\int_{M}\eta^2\psi_a^{2t-1}\mathcal{L}_{p,w}\psi_a=&-(2t-1)\int_{M}H^{p-2}\hin{A_{w}\left(\nabla\psi_a\right)}{\nabla\psi_a}
\psi_a^{2t-2}\eta^2-2\int_MH^{p-2}\hin{A_{w}\left(\nabla\psi_a\right)}{\nabla\eta}\eta\psi_a^{2t-1}\\
=&-\dfrac{2t-1}{t^2}\int_{M}H^{p-2}\hin{A_{w}\left(\nabla\psi_a^t\right)}{\nabla\psi_a^t}\eta^2-\dfrac{2}{t}\int_MH^{p-2}
\hin{A_{w}\left(\nabla\psi_a^t\right)}{\nabla\eta}\eta\psi_a^{t}\\
=&-\dfrac{t-1}{t^2}\int_{M}H^{p-2}\hin{A_{w}\left(\nabla\psi_a^t\right)}{\nabla\psi_a^t}\eta^2\\
&+\dfrac{1}{t}\int_{M}H^{p-2}\hin{A_{w}\left(\nabla\eta\right)}{\nabla\eta}\psi_a^{2t}-\dfrac{1}{t}\int_{M}H^{p-2}
\hin{A_{w}\left(\nabla\left(\eta\psi_a^t\right)\right)}{\nabla\left(\eta\psi_a^t\right)}\\
\leq&\dfrac{1}{t}\int_{M}H^{p-2}\hin{A_{w}\left(\nabla\eta\right)}{\nabla\eta}\psi_a^{2t}-\dfrac{1}{t}\int_{M}H^{p-2}
\hin{A_{w}\left(\nabla\left(\eta\psi_a^t\right)\right)}{\nabla\left(\eta\psi_a^t\right)}.
\end{align*}
Therefore,
\begin{align*}
\tau_0 C_{n,p,\kappa,\delta_0^+,a}\int_MH^{p-2}\psi_a^{2t}\eta^2\leq\dfrac{1}{t}\int_{M}H^{p-2}
\hin{A_{w}\left(\nabla\eta\right)}{\nabla\eta}\psi_a^{2t}-\dfrac{1}{t}\int_{M}H^{p-2}\hin{A_{w}
\left(\nabla\left(\eta\psi_a^t\right)\right)}{\nabla\left(\eta\psi_a^t\right)}.
\end{align*}
In particular, for some positive constant $C$
\begin{align*}
\int_MH^{p-2}\psi_a^{2t}\eta^2\leq\dfrac{C}{t}\int_{M}H^{p-2}\psi_a^{2t}\abs{\nabla\eta}^2,\quad\forall t\geq1.
\end{align*}
Choose cutoff functions $\eta_i$ with the properties:
\begin{align*}
\psi_i\in C_0^{\infty}\left(B_{i+1}\right),\quad 0\leq\eta_i\leq1,\quad\eta_i\vert_{B_{i}}=1,\quad\abs{\nabla\eta_i}\leq 2,\quad i=0,1,2,\dots.
\end{align*}
We get
\begin{align*}
\int_{B_{i}}H^{p-2}\psi_a^{2t}\leq\dfrac{C}{t}\int_{B_{i+1}}H^{p-2}\psi_a^{2t},\quad i=0,1,2,\dots.
\end{align*}
Consequently,
\begin{align*}
\int_{B_i}H^{p-2}\psi_a^{2t}\leq\left(\dfrac{C}{t}\right)^{j-i}\int_{B_k}H^{p-2}\psi_a^{2t},\quad\forall j>i.
\end{align*}
Since the Ricci curvature is bounded from below by $-(n-1)\kappa$, applying the Bishop-Gromov volume comparison theorem, we get
\begin{align*}
\mathrm{Vol}\left(B_{R}\right)\leq e^{C_n\left(1+\sqrt{\kappa}R\right)}.
\end{align*}
We obtain for every $t\geq 1$
\begin{align*}
\int_{B_i}H^{p-2}\psi_a^{2t}\leq e^{C_n}\left(\dfrac{C}{t}\right)^{-i}\left(\dfrac{C}{t}e^{C_n\sqrt{\kappa}}\right)^{j}\Lambda^{p-2+2t},\quad\forall j>i.
\end{align*}
Taking $t=\max\set{1,\, Ce^{C_n\sqrt{\kappa}}}$ and letting $j\to\infty$ and then $i\to\infty$, we obtain
\begin{align*}
\int_MH^{p-2}\psi_{a}^{2t}=0.
\end{align*}
which implies
\begin{align*}
\sup_MH\leq a.
\end{align*}
According to the arbitrariness of $a\in(\Lambda_0,\, \Lambda]$, we conclude that
\begin{align*}
\sup_MH\leq\Lambda_0.
\end{align*}
We complete the proof.
\end{proof}

\section{Liouville theorems}

We begin with the following simple case, i.e., $M^n$ is a compact manifold without boundary.
\begin{theorem}
Let $M^n$ be a compact Riemannian manifold with dimension $n$ and nonnegative Ricci curvature. Assume  for some constant $\alpha$,
\begin{align*}
p>1,\quad \alpha\left(1-\dfrac{(n-p)\alpha}{(n+1)p-n}\right)>0,\quad \alpha f(t)-tf'(t)\geq0,\quad\forall t>0.
\end{align*}
Then every positive and weak solution $u$ to \eqref{eq:quasi-h} is constant.
\end{theorem}

\begin{proof}
Let $u$ be a positive and weak solution to \eqref{eq:quasi-h}. By \eqref{eq:basic'} we know that the following inequality holds true a.e. on $M^n$
\begin{align*} \mathrm{div}\left(u^{\frac{(n-1)p\alpha}{(n+1)p-n}}\left(\nabla_{\mathbf{X}}\mathbf{X}-\dfrac{\mathrm{div}\mathbf{X}}{n}\mathbf{X}\right)\right)
\geq 0.
\end{align*}
On the other hand, by Stokes theorem we have
$$\int_{M^n}\mathrm{div}\left(u^{\frac{(n-1)p\alpha}{(n+1)p-n}}\left(\nabla_{\mathbf{X}}\mathbf{X}-\dfrac{\mathrm{div}\mathbf{X}}{n}\mathbf{X}\right)\right)
= 0,$$
since $M^n$ is a compact manifold without boundary. This implies
$$\mathrm{div}\left(u^{\frac{(n-1)p\alpha}{(n+1)p-n}}\left(\nabla_{\mathbf{X}}\mathbf{X}-\dfrac{\mathrm{div}\mathbf{X}}{n}\mathbf{X}\right)\right)
= 0, \quad \mbox{a.e. on}\,\, M^n.$$
It is easy to see from \eqref{eq:basic'} that the above equality holds only if
\begin{align*}
\alpha\left(1-\dfrac{(n-p)\alpha}{(n+1)p-n}\right)=0,\quad f(u)=Au^{\alpha},\quad Ric\left(\nabla u,\nabla u\right)=0, \quad \tr\left(\mathscr X^2\right)=0,
\end{align*}
whenever $\nabla u\neq0$ and $A$ is a constant. This implies $\nabla u=0$ almost everywhere on $M^n$. Thus, $u$ must be a constant function.
\end{proof}

Another application of \cref{lem:basic2} is that we provide a new and simple proof of Bidaut-Veron and Veron's Liouville result \cite[Theorem 6.1]{BidMarVer91nonlinear} on the following equation
$$\Delta u-\lambda u+u^q=0$$
defined on a compact manifold without boundary.

\begin{theorem}[{\cite[Theorem 6.1]{BidMarVer91nonlinear}}]\label{thm:BV}
Let $(M,g)$ be a closed manifold with dimension $n\geq 2$ and $u$ be a positive and weak solution to
\begin{align*}
\Delta u-\lambda u+u^q=0,
\end{align*}
where $q>1$ and $\lambda>0$. Assume that the Ricci curvature satisfies
\begin{align}\label{cond:veron1}
\mathrm{Ric}_g\geq\frac{n-1}{n}(q-1)\lambda g
\end{align}
and
\begin{align}\label{cond:veron2}
q\leq \frac{n+2}{n-2}.
\end{align}
Furthermore, assume also that one of the condition \eqref{cond:veron1} and \eqref{cond:veron2} is strict if $(M,g)=(\mathbb S^n, g_0)$. Then $u$ is a constant with value $\lambda^{\frac{1}{q-1}}$.
\end{theorem}

\begin{proof}
Substituting  $p=2$ and $f=-\lambda u+u^q$ into \eqref{eq:basic-0} and integrating by parts, we obtain
\begin{align*}
 0
=&\int_Mu^{a} \tr\left(\mathscr X^2\right)+u^{a+2b}\mathrm{Ric}\left(\nabla u,\nabla u\right) \\
&+\int_M\left(\dfrac{ a}{ 2}\left(1-\dfrac{(n-2)a}{2(n-1)}\right)-\dfrac{n-1}{n}\left(b+\dfrac{na}{2(n-1)}\right)^2\right)u^{a+2b-2}
\abs{\nabla u}^{4}\\
&+\int_M\dfrac{n-1}{n}\left(1-\dfrac{(n+2)a}{2(n-1) } \right)\lambda  u^{a+2b}\abs{\nabla u}^{2}+\dfrac{n-1}{n}\left( \dfrac{(n+2)a}{2(n-1) } -q\right)   u^{a+2b+q-1}\abs{\nabla u}^{2}.
\end{align*}
Now, in the above integral identity we choose
$$a=\frac{2(n-1)}{n+2}q\quad\mbox{and}\quad b=-\frac{2q}{n+2},$$
then, it follows that
\begin{align*}
 0
=&\int_Mu^{\frac{2(n-1)}{n+2}q} \tr\left(\mathscr X^2\right)+\int_Mu^{\frac{2(n-3)}{n+2}q}\left(\mathrm{Ric}\left(\nabla u,\nabla u\right) -\dfrac{n-1}{n}\left(q-1\right)\lambda   \abs{\nabla u}^{2}\right)\\
&+\int_M \frac{ n-1 }{n+2}q\left(1-\frac{n-2}{n+2}q\right) u^{\frac{2(n-3)}{n+2}q-2}
\abs{\nabla u}^{4}.
\end{align*}
On the other hand, by the assumptions in this corollary we know that each of the three terms on the right hand side of the above is non-negative. Next, we need only to consider the following three cases.
\begin{enumerate}[\text{Case} 1)]
\item If $q<\frac{n+2}{n-2}$, then it is easy to see that
$$0\geq \int_M \frac{ n-1 }{n+2}q\left(1-\frac{n-2}{n+2}q\right) u^{\frac{2(n-3)}{n+2}q-2}\abs{\nabla u}^{4}$$
implies $\nabla u=0$ and $u$ is a constant.

\item If $q=\frac{n+2}{n-2}$, $(M,g)=(\mathbb{S}^n, g_0)$ and \eqref{cond:veron1} is a strict inequality, then
$$0\geq \int_Mu^{\frac{2(n-3)}{n-2} }\left(\mathrm{Ric}\left(\nabla u,\nabla u\right) -\dfrac{n-1}{n}\left(q-1\right)\lambda\abs{\nabla u}^{2}\right)$$
implies $\nabla u=0$ and hence $u$ is a constant.

\item If $q=\frac{n+2}{n-2}$, $(M,g)\neq (\mathbb{S}^n, g_0)$ and \eqref{cond:veron1} is not strict, then we have
$$\int_Mu^{\frac{2(n-1)}{n-2}} \tr\left(\mathscr X^2\right)=0.$$
Hence, the fact $u>0$ forces $\tr\left(\mathscr X^2\right)=0$, i.e.,
$$\mathrm{tr}\left(\mathcal X-\frac{\mathrm{tr}\mathcal X}{n}\right)^2=0.$$
Noting
$$
\mathcal X=\nabla \left(u^b\nabla u\right) =
\nabla \left(u^{-\frac{2}{n-2}}\nabla u\right) =
\begin{cases}
 \frac{n-2}{n-4}\nabla\nabla \left(u^{\frac{n-4}{n-2}}\right) , &\text{ if } n\neq 4;
 \\
\nabla \nabla\ln u  , &\text{ if }  n=4,
\end{cases}
$$
hence, we have
\begin{align*}
\nabla\nabla \left(u^{\frac{n-4}{n-2}}\right)=\frac{1}{n}\Delta \left(u^{\frac{n-4}{n-2}}\right)g, \quad \text{ if } n\neq 4;
\\
\nabla\nabla (\ln u)=\frac{1}{n}\Delta (\ln u)g, \quad \text{ if } n= 4.
\end{align*}
According to the rigidity theorem \cite[Theorem E]{YanOba70conformal} (see \cite{IshTas59riemannian} for more detailed discussions), the equation $$\nabla \nabla f=\frac{1}{n}\Delta f g$$
admits a non-trivial solution if and only if $(M,g)=(\mathbb S^n, g_0)$. Thus $u$ must be a constant.
\end{enumerate}
Thus we complete the proof of this theorem.
\end{proof}

From now on, without loss of generality, we always assume $M$ is noncompact. We have the following fundamental integral estimate.
\begin{lem}\label{lem:crucial-0}
Let $M^n$ be a  Riemannian manifold with dimension $n$ and nonnegative Ricci curvature.  Assume  for some constant $\alpha$,
\begin{align*}
p>1,\quad \alpha\left(1-\dfrac{(n-p)\alpha}{(n+1)p-n}\right)>0,\quad \alpha f(t)-tf'(t)\geq0,\quad\forall t>0.
\end{align*}
If $u$ is a global positive and weak solution to \eqref{eq:quasi-h}, then for any $\eta\in C^\infty_0(M)$ and every constant $\gamma\geq2p$ there holds
\begin{align}\label{eq:crucial-0}
\int_Mu^{-\frac{((n+1)p-2n)\alpha}{(n+1)p-n}-2}\abs{\nabla u}^{2p}\eta^{\gamma}+\int_{M}f(u)^2u^{-\frac{((n+1)p-2n)\alpha}{(n+1)p-n}}\eta^{\gamma}
\leq& C(n,p,\alpha)\gamma^{2p}\int_Mu^{-\frac{((n+1)p-2n)\alpha}{(n+1)p-n}+2(p-1)}\abs{\nabla\eta}^{2p}\eta^{\gamma-2p}.
\end{align}
Here $C(n,p,\alpha)$ is a positive constant depending only on $n$, $p$, and $\alpha$.
\end{lem}

\begin{proof}
It follows from \eqref{eq:basic'} in \autoref{lem:basic2}, the following inequality holds
\begin{align}\label{eq:liouville-1}
\begin{split}
\Div\left(u^{\frac{(n-1)p\alpha}{(n+1)p-n}} \mathscr{X} (\mathbf X)\right)\geq\,&u^{\frac{(n-1)p\alpha}{(n+1)p-n}}\left(\tr\left(\mathscr X^2\right)\right)+A_{n,p,\alpha}u^{-\frac{((n+1)p-2n)\alpha}{(n+1)p-n}-2}\abs{\nabla u}^{2p}\\
&+\dfrac{n-1}{n}\left(\alpha f(u)-uf'(u)\right)u^{-\frac{((n+1)p-2n)\alpha}{(n+1)p-n}-1}\abs{\nabla u}^{p}
\end{split}
\end{align}
in the $L^1_{loc}$ sense, where
\begin{align*}
\mathbf{X}=u^{-\frac{n(p-1)\alpha}{(n+1)p-n}}\abs{\nabla u}^{p-2}\nabla u\quad\quad \mbox{and}\quad\quad
\mathscr X=\nabla\mathbf X-\frac{\Div\mathbf X}{n}id_{TM}\in\mathrm{End}(TM).
\end{align*}

For the sake of simplicity, we denote
\begin{align*}
t=\dfrac{((n+1)p-2n)\alpha}{(n+1)p-n}.
\end{align*}
For any nonnegative function $\eta\in C_0^{\infty}\left(M\right)$ and positive number $\gamma\geq 2p$, it follows from \eqref{eq:liouville-1}
\begin{align}\label{eq:liouville-2}
\int_M\Div\left(u^{\frac{(n-1)p\alpha}{(n+1)p-n}}\mathscr X(\mathbf X)\right)\eta^{\gamma}\geq&\int_Mu^{\frac{(n-1)p\alpha}{(n+1)p-n}}\left(\tr\left(\mathscr X^2\right)\right)\eta^{\gamma}+A_{n,p,\alpha}\int_Mu^{-t-2}\abs{\nabla u}^{2p}\eta^{\gamma}.
\end{align}
Direct computation shows
\begin{align*}
 \int_M\Div\left(u^{\frac{(n-1)p\alpha}{(n+1)p-n}} \mathscr X(\mathbf{X})\right) \eta^{\gamma}
=&-\gamma\int_Mu^{\frac{(n-1)p\alpha}{(n+1)p-n}}\hin{\mathscr X(\mathbf{X})}{\nabla\eta}\eta^{\gamma-1}
\leq \gamma\int_Mu^{\frac{(n-1)p\alpha}{(n+1)p-n}}\abs{ \mathscr X(\mathbf{X})}\abs{\nabla\eta}\eta^{\gamma-1}
\end{align*}
Since $p>1$, applying \autoref{lem:basic-3}, we have
\begin{align*}
\abs{\mathscr X(\mathbf X)}\leq&\abs{\mathscr X}\abs{\mathbf{X}}
\leq \left(\dfrac{1+(p-1)^2}{2(p-1)}\tr\left(\mathscr X^2\right)\right)^{1/2}u^{-\frac{n(p-1)\alpha}{(n+1)p-n}}\abs{\nabla u}^{p-1}.
\end{align*}
It follows
\begin{align*}
 \int_M\Div\left(u^{\frac{(n-1)p\alpha}{(n+1)p-n}} \mathscr X(\mathbf{X})\right) \eta^{\gamma}
\leq&\sqrt{\frac{1+(p-1)^2}{2(p-1)}}\gamma\int_Mu^{\frac{(n-1)p\alpha}{(n+1)p-n}}\left(\tr\left(\mathscr X^2\right)\right)^{1/2}u^{-\frac{n(p-1)\alpha}{(n+1)p-n}}\abs{\nabla u}^{p-1}\abs{\nabla\eta}\eta^{\gamma-1}\\
=&\sqrt{\frac{1+(p-1)^2}{2(p-1)}}\gamma\int_Mu^{\frac{(n-p)\alpha}{(n+1)p-n}}\left(\tr\left(\mathscr X^2\right)\right)^{1/2}\abs{\nabla u}^{p-1}\abs{\nabla\eta}\eta^{\gamma-1}\\.
\end{align*}
Using H\"older's inequality with respective exponent pair
\begin{align*}
\left(2,\,\dfrac{2p}{p-1},\, 2p\right),
\end{align*}
we have
\begin{align*}
 \int_M\Div\left(u^{\frac{(n-1)p\alpha}{(n+1)p-n}} \mathscr X(\mathbf{X})\right) \eta^{\gamma}
\leq&\sqrt{\frac{1+(p-1)^2}{2(p-1)}}\gamma\left(\int_{M}u^{\frac{(n-1)p\alpha}{(n+1)p-n}}\left(\tr\left(\mathscr X^2\right)\right)\eta^{\gamma}\right)^{1/2}\left(\int_Mu^{-\frac{((n+1)p-2n)\alpha}{(n+1)p-n}-2}\abs{\nabla u}^{2p}\eta^{\gamma}\right)^{\frac{p-1}{2p}}\\
&\times\left(\int_Mu^{-\frac{((n+1)p-2n)\alpha}{(n+1)p-n}+2(p-1)}\abs{\nabla\eta}^{2p}\eta^{\gamma-2p}\right)^{\frac{1}{2p}},
\end{align*}
which implies
\begin{align}\label{eq:liouville-5}
\begin{split}&\int_M\Div\left(u^{\frac{(n-1)p\alpha}{(n+1)p-n}}\mathscr X(\mathbf X)\right)\eta^{\gamma}\\
\leq&\sqrt{\frac{1+(p-1)^2}{2(p-1)}}\gamma\left(\int_{M}u^{\frac{(n-1)p\alpha}{(n+1)p-n}}\left(\tr\left(\mathscr X^2\right)\right)\eta^{\gamma}\right)^{1/2}\left(\int_Mu^{-t-2}\abs{\nabla u}^{2p}\eta^{\gamma}\right)^{\frac{p-1}{2p}}\left(\int_Mu^{-t+2(p-1)}\abs{\nabla\eta}^{2p}\eta^{\gamma-2p}\right)^{\frac{1}{2p}}.
\end{split}
\end{align}
Using Young's inequality with the respective exponent pair
\begin{align*}
\left(2,\, \dfrac{2p}{p-1},\, 2p\right),
\end{align*}
we obtain from \eqref{eq:liouville-2} and \eqref{eq:liouville-5}
\begin{align*}
&\int_Mu^{\frac{(n-1)p\alpha}{(n+1)p-n}}\left(\tr\left(\mathscr X^2\right)\right)\eta^{\gamma}+A_{n,p,\alpha}\int_Mu^{-t-2}\abs{\nabla u}^{2p}\eta^{\gamma}\\
\leq&\dfrac12\int_Mu^{\frac{(n-1)p\alpha}{(n+1)p-n}}\left(\tr\left(\mathscr X^2\right)\right)\eta^{\gamma}+\dfrac{p-1}{2p}A_{n,p,\alpha}\int_Mu^{-t-2}\abs{\nabla u}^{2p}\eta^{\gamma}\\
&+\dfrac{1}{2p}\left(\frac{1+(p-1)^2}{2(p-1)}\right)^{p}A_{n,p,\alpha}^{1-p}\gamma^{2p}\int_Mu^{-t+2(p-1)}\abs{\nabla\eta}^{2p}
\eta^{\gamma-2p}
\end{align*}
which implies
\begin{align}\label{eq:liouville-7}
\int_Mu^{-t-2}\abs{\nabla u}^{2p}\eta^{\gamma}
\leq&\left(\frac{1+(p-1)^2}{2(p-1)A_{n,p,\alpha}}\right)^{p}\gamma^{2p}\int_Mu^{-t+2(p-1)}\abs{\nabla\eta}^{2p}\eta^{\gamma-2p}.
\end{align}

On the other hand,
\begin{align}\label{eq:liouville-3}
\begin{split}
\int_{M}f(u)^2u^{-t}\eta^{\gamma}=&\int_M\left(-\Delta_pu\right)f(u)u^{-t}\eta^{\gamma}\\
=&\int_M\abs{\nabla u}^{p}\left(uf'(u)-tf(u)\right)u^{-t-1}\eta^{\gamma}+\gamma\int_M\abs{\nabla u}^{p-2}f(u)u^{-t}\hin{\nabla u}{\nabla\eta}\eta^{\gamma-1}.
\end{split}
\end{align}
Using H\"older's inequality with respective exponent pair
\begin{align*}
\left(2,\,\dfrac{2p}{p-1},\,2p\right),
\end{align*}
we get
\begin{align}\label{eq:liouville-6}
\begin{split}
\gamma\int_M\abs{\nabla u}^{p-2}f(u)u^{-t}\hin{\nabla u}{\nabla\eta}\eta^{\gamma-1}\leq&\gamma\int_Mf(u)u^{-t}\abs{\nabla u}^{p-1}\abs{\nabla\eta}\eta^{\gamma-1}\\
\leq&\gamma\left(\int_{M}f(u)^2u^{-t}\eta^{\gamma}\right)^{\frac{1}{2}}\left(\int_M\abs{\nabla u}^{2p}u^{-t-2}\eta^{\gamma}\right)^{\frac{p-1}{2p}}\left(\int_{M}u^{-t+2(p-1)}\abs{\nabla \eta}^{2p}\eta^{\gamma-2p}\right)^{\frac{1}{2p}}.
\end{split}
\end{align}
It follows from \eqref{eq:liouville-3} and \eqref{eq:liouville-6}
\begin{align*}
\int_{M}f(u)^2u^{-t}\eta^{\gamma}=&\int_M\abs{\nabla u}^{p}\left(uf'(u)-tf(u)\right)u^{-t-1}\eta^{\gamma}+\gamma\int_M\abs{\nabla u}^{p-2}f(u)u^{-t}\hin{\nabla u}{\nabla\eta}\eta^{\gamma-1}\\
\leq&\left(\alpha-t\right)\int_M\abs{\nabla u}^{p}f(u)u^{-t-1}\eta^{\gamma}\\
&+\gamma\left(\int_{M}f(u)^2u^{-t}\eta^{\gamma}\right)^{\frac{1}{2}}\left(\int_M\abs{\nabla u}^{2p}u^{-t-2}\eta^{\gamma}\right)^{\frac{p-1}{2p}}\left(\int_{M}u^{-t+2(p-1)}\abs{\nabla \eta}^{2p}\eta^{\gamma-2p}\right)^{\frac{1}{2p}}\\
\leq&\abs{\alpha-t}\left(\int_M\abs{f(u)}^2u^{-t}\eta^{\gamma}\right)^{1/2}\left(\int_M\abs{\nabla u}^{2p}u^{-t-2}\eta^{\gamma}\right)^{1/2}\\
&+\gamma\left(\int_{M}f(u)^2u^{-t}\eta^{\gamma}\right)^{\frac{1}{2}}\left(\int_M\abs{\nabla u}^{2p}u^{-t-2}\eta^{\gamma}\right)^{\frac{p-1}{2p}}\left(\int_{M}u^{-t+2(p-1)}\abs{\nabla \eta}^{2p}\eta^{\gamma-2p}\right)^{\frac{1}{2p}}.
\end{align*}
We obtain
\begin{align*}
\int_{M}f(u)^2u^{-t}\eta^{\gamma}\leq&2\abs{\alpha-t}^2\int_M\abs{\nabla u}^{2p}u^{-t-2}\eta^{\gamma}+2\gamma^2\left(\int_M\abs{\nabla u}^{2p}u^{-t-2}\eta^{\gamma}\right)^{\frac{p-1}{p}}\left(\int_{M}u^{-t+2(p-1)}\abs{\nabla \eta}^{2p}\eta^{\gamma-2p}\right)^{\frac{1}{p}}.
\end{align*}
Hence, we conclude from the above inequality together with \eqref{eq:liouville-7}
\begin{align*}
\int_{M}f(u)^2u^{-t}\eta^{\gamma}\leq& C(n,p,\alpha)\gamma^{2p}\int_Mu^{-t+2(p-1)}\abs{\nabla\eta}^{2p}\eta^{\gamma-2p}.
\end{align*}
Immediately, the required inequality \eqref{eq:crucial-0} follows from \eqref{eq:liouville-7} and the above inequality and we finish the proof of this lemma.
\end{proof}

We need also the following weak Harnack inequality for positive and weak $p$-superharmonic function, the proof of which is given in the appendix A of this paper.
\begin{lem}\label{lem:weak-hanack}
Let $M^n$ be a complete Riemannian manifold with dimension $n$ and nonnegative Ricci curvature. If $p>1$ and $u$ is a positive and weak p-superharmonic function, then for all $q\in\left(0,\, (p-1)\chi_{n,p}\right)$ there exists a uniform constant $C=C(n,p,\chi_{n,p},q)$ such that
\begin{align}\label{eq:weak-hanack}
\inf_{B_R}u\geq C R^{-n/q}\norm{u}_{L^{q}\left(B_{2R}\right)},\quad\forall R>0.
\end{align}
Here $\chi_{n,p}=\frac{n}{n-p}$ if $1< p<n$, while $\chi_{n,p}$ can be any positive number which is larger than $1$ if $p\geq n$.
\end{lem}

\begin{theorem}\label{thm:integral-estimate1}
Let $M^n$ be a complete Riemannian manifold with dimension $n$ and nonnegative Ricci curvature. Assume for some constant $\alpha$
\begin{align*}
p>1,\quad \dfrac{1}{\alpha}>\dfrac{n-p}{(n+1)p-n},\quad \alpha f(t)-tf'(t)\geq0,\quad\forall t>0.
\end{align*}
If we assume additionally that $f$ is nonnegative and
\begin{align*}
\rho\coloneqq \dfrac{((n+1)p-2n)\alpha}{((n+1)p-n)(p-1)\chi_{n,p}}-\dfrac{2}{\chi_{n,p}}>-1,
\end{align*}
then for any $q\in\left(0,\, \min\set{2,\frac{2}{1+\rho^+}}\right)$,
\begin{align}\label{eq:integral-estimate1}
\norm{f(u)u^{1-p}}_{L^{q}\left(B_{R}\right)}\leq C R^{\frac{n}{q}-p},\quad\forall R>0.
\end{align}
Here the constant $C$ depends only on $n, p, \chi_{n,p}, \alpha$ and $q$.
\end{theorem}

\begin{proof}
As before, we denote
\begin{align*}
t=\dfrac{((n+1)p-2n)\alpha}{(n+1)p-n}.
\end{align*}
The assumption $\rho>-1$ is equivalent to
\begin{align*}
2(p-1)-t<(p-1)\chi_{n,p}.
\end{align*}

For each $q\in(0,2)$ and denote by
\begin{align*}
s=t-(p-1)q,
\end{align*}
it follows from \eqref{eq:crucial-0}
\begin{align*}
\int_{M}f(u)^{q}u^{-(p-1)q}\eta^{\gamma}=&\int_{M}f(u)^{q}u^{s-t}\eta^{\gamma}\\
\leq&\left(\int_Mf(u)^2u^{-t}\eta^{\gamma}\right)^{\frac{q}{2}}\left(\int_Mu^{\frac{2s}{2-q}-t}\eta^{\gamma}\right)^{\frac{2-q}{2}}\\
\leq&C\left(\int_Mu^{2(p-1)-t}\abs{\nabla\eta}^{2p}\eta^{\gamma-2p}\right)^{\frac{q}{2}}\left(\int_Mu^{\frac{2s}{2-q}-t}
\eta^{\gamma}\right)^{\frac{2-q}{2}}\\
=&C\left(\int_Mu^{2(p-1)-t}\abs{\nabla\eta}^{2p}\eta^{\gamma-2p}\right)^{\frac{q}{2}}\left(\int_Mu^{\frac{q}{2-q}\left(t-2(p-1)\right)}
\eta^{\gamma}\right)^{\frac{2-q}{2}},
\end{align*}
where $C=C(n,p,\chi_{n,p},\alpha,q)$. Hence, by using the properties of the cutoff functions we infer
\begin{align}\label{eq:f^q}
\fint_{B_{R/2}}f(u)^{q}u^{-(p-1)q}\leq &CR^{-pq}\left(\fint_{B_R}u^{2(p-1)-t}\right)^{\frac{q}{2}}\left(\fint_{B_R}u^{\frac{q}{2-q}\left(t-2(p-1)\right)}\right)^{\frac{2-q}{2}}.
\end{align}

Firstly, we consider a special case:
\begin{align*}
t=2(p-1).
\end{align*}
Using the Bishop-Gromov volume comparison theorem, we have from \eqref{eq:f^q}
\begin{align*}
\int_{B_{R/2}}f(u)^{q}u^{-(p-1)q}\leq &CR^{n-pq}.
\end{align*}

Secondly, we consider the case:
\begin{align*}
t< 2(p-1).
\end{align*}
By assumption
\begin{align*}
0<2(p-1)-t<(p-1)\chi_{n,p}.
\end{align*}
We can use \autoref{lem:weak-hanack} and the Bishop-Gromov comparison theorem to conclude from \eqref{eq:f^q}
\begin{align*}
\fint_{B_{R/2}}f(u)^{q}u^{-(p-1)q}\leq &CR^{-pq}\left(\fint_{B_R}u^{2(p-1)-t}\right)^{\frac{q}{2}}\left(\left(\inf_{B_R}u\right)^{\frac{q}{2-q}
\left(t-2(p-1)\right)}\right)^{\frac{2-q}{2}}\\
\leq&CR^{-pq}\left(\fint_{B_R}u^{2(p-1)-t}\right)^{\frac{q}{2}}\left(\left(R^{-n}\int_{B_{2R}}u^{2(p-1)-t}\right)^{\frac{1}{2(p-1)-t}} \right)^{\frac{q}{2}\left(t-2(p-1)\right)}\\
\leq&CR^{-pq}\left(\frac{1}{V(B_{2R})}\int_{B_{2R}}u^{2(p-1)-t}\right)^{\frac{q}{2}}\left(R^{-n} \int_{B_{2R}}u^{2(p-1)-t} \right)^{-\frac{q}{2} }\\
\leq &CR^{\frac{qn}{2}-pq}V(B_{2R})^{-\frac{q}{2}}\\
\leq &CR^{\frac{qn}{2}-pq}V(B_{R/2})^{-\frac{q}{2}}.
\end{align*}
The constant $C$ may differ line by line in the above calculation, but they only depend on $n$ and $p$. Since $q<2$, we have
\begin{align*}
\int_{B_{R/2}}f(u)^{q}u^{s-t}\leq CR^{\frac{qn}{2}-pq}V(B_{R/2})^{1-\frac{q}{2}}\leq  CR^{n-pq}.
\end{align*}

Finally, we consider the case
\begin{align*}
t>2(p-1).
\end{align*}
If
\begin{align*}
\dfrac{q}{2-q}\left(t-2(p-1)\right)<(p-1)\chi_{n,p},
\end{align*}
which is equivalent to
\begin{align*}
q<\dfrac{2}{1+\rho}=\min\set{2,\, \dfrac{2}{1+\rho^+}},
\end{align*}
then we also use the weak Harnack inequality \eqref{eq:weak-hanack} for $p$-superharmonic function and the Bishop-Gromov comparison theorem to conclude from \eqref{eq:f^q}
\begin{align*}
\fint_{B_{R/2}}f(u)^{q}u^{-(p-1)q}\leq &CR^{-pq}\left(\left(\inf_{B_R}u\right)^{2(p-1)-t}\right)^{\frac{q}{2}}\left(\fint_{B_R}u^{\frac{q}{2-q}\left(t-2(p-1)\right)}\right)^{\frac{2-q}{2}}\\
\leq&CR^{-pq}\left(\left(\fint_{B_R}u^{\frac{q}{2-q}\left(t-2(p-1)\right)}\right)^{\frac{2-q}{q(t-2(p-1))}}\right)^{\frac{q(2(p-1)-t)}{2}}
\left(\fint_{B_R}u^{\frac{q}{2-q}\left(t-2(p-1)\right)}\right)^{\frac{2-q}{2}}\\
\leq&CR^{-pq}.
\end{align*}
Thus,
\begin{align*}
\int_{B_{R/2}}f(u)^{q}u^{-(p-1)q}\leq CR^{n-pq}.
\end{align*}
We complete the proof, i.e., we obtain the desired integral estimate \eqref{eq:integral-estimate1}.
\end{proof}

Using this integral estimate, we can prove \autoref{thm:main2}.
\begin{proof}[Proof of \autoref{thm:main2}]
According to \autoref{main:cor-1}, we may assume
\begin{align*}
\alpha\geq\dfrac{(n+3)(p-1)}{n-1}.
\end{align*}
Moreover, if $1<p<n$, then
\begin{align*}
0<\alpha<\dfrac{(n+1)p-n}{n-p}.
\end{align*}
Denote by
\begin{align*}
\rho=\dfrac{((n+1)-2n)\alpha}{((n+1)p-n)(p-1)\chi_{n,p}}-\dfrac{2}{\chi_{n,p}}.
\end{align*}

We need to deal with the following five cases.
\medskip

{\bf Case 1. $p\geq n$.}

In this case, since $p>\frac{n}{2}$, we can choose first $q\in(0,2)$ and then $\chi_{n,p}$ large such that
\begin{align*}
n-pq<0,\quad \abs{\rho}<1.
\end{align*}
Then \autoref{thm:integral-estimate1} implies that
\begin{align*}
\int_{B_R}f(u)^qu^{(1-p)q}\leq CR^{n-pq}\to0,\quad\text{as}\ R\to\infty.
\end{align*}
We conclude that $f(u)\equiv0$ and $u$ is $p$-harmonic and $u$ is constant. This prove the first part $(A)$.
\medskip

{\bf Case 2. $3\leq n<2p<2n$ and
\begin{align*}
\alpha\leq\dfrac{2((n+1)p-n)(p-1)}{(n+1)p-2n}.
\end{align*}}

In this case, we must have
\begin{align*}
0 \leq 2(p-1)-\dfrac{((n+1)p-2n)\alpha}{(n+1)p-n}<(p-1)\chi_{n,p},
\end{align*}
since
$$\chi_{n,p}=\frac{n}{n-p}.$$
Applying \autoref{thm:integral-estimate1}, we can verify that for each $0<q<2$ there holds
\begin{align*}
\int_{B_R}f(u)^qu^{(1-p)q}\leq CR^{n-pq}.
\end{align*}
Since $n<2p$, one can take $q\in(0,\,2)$ to conclude that
\begin{align*}
\int_Mf(u)u^{(1-p)q}=0.
\end{align*}
Consequently, $u$ is constant.
\medskip

{\bf Case 3. $\frac{4}{3}<p<n=2$ and
\begin{align*}
\alpha\leq\dfrac{2(3p-2)(p-1)}{3p-4}.
\end{align*}}

For the present situation, we can get
\begin{align*}
0 \leq 2(p-1)-\dfrac{((n+1)p-2n)\alpha}{(n+1)p-n}<(p-1)\chi_{n,p}.
\end{align*}
By taking a similar argument to the case 2, we conclude that $u$ is a constant function.
\medskip

{\bf Case 4. $1<p\leq\frac{4}{3}$, $n=2$ and
\begin{align*}
\alpha<\dfrac{2(p-1)^2(3p-2)}{(4-3p)(2-p)}.
\end{align*}}

In this case, we have
\begin{align*}
0 \leq 2(p-1)-\dfrac{((n+1)p-2n)\alpha}{(n+1)p-n}<(p-1)\chi_{n,p}.
\end{align*}
Hence, we conclude that $u$ is a constant function.
\medskip

{\bf Case 5. $\frac{2n}{n+1}< p<n$ and
\begin{align*}
\alpha>\dfrac{2((n+1)p-n)(p-1)}{(n+1)p-2n}.
\end{align*}}

In this case, since
$$0<\alpha<\frac{(n+1)p-n}{n-p},$$
we must have
\begin{align*}
0<\dfrac{((n+1)p-2n)\alpha}{(n+1)p-n}-2(p-1)<\dfrac{(2p-n)(p-1)}{n-p}.
\end{align*}
Hence, it follows $2p>n$ and we can take $q\in\left(\frac{n}{p},\,2\right)$ such that
\begin{align*}
0<\dfrac{((n+1)p-2n)\alpha}{(n+1)p-n}-2(p-1)<\dfrac{2-q}{q}\cdot\dfrac{n(p-1)}{n-p}.
\end{align*}
In particular, we have
\begin{align*}
n<pq \quad \mbox{and} \quad 0<q<\dfrac{2}{1+\rho}.
\end{align*}
Applying \autoref{thm:integral-estimate1}, we conclude that $u$ is a constant function. Combing the case 2 and the case 4, we prove the second part $(B)$. Combing the cases 3-5, we prove the third part $(C)$.
\end{proof}

Finally, we give a proof of \autoref{thm:main3}.

\begin{proof}[Proof of \autoref{thm:main3}]
Let $u$ be a positive and weak solution to \eqref{eq:quasi-h}. According to \autoref{lem:crucial-0}, for every $\gamma\geq2p$ and nonnegative  function $\eta\in C_0^{\infty}\left(M\right)$
\begin{align}\label{add}
\int_{M}f(u)^2u^{-t}\eta^{\gamma}
\leq& C(n,p,\alpha)\gamma^{2p}\int_Mu^{2(p-1)-t}\abs{\nabla\eta}^{2p}\eta^{\gamma-2p},
\end{align}
where
\begin{align*}
t=\dfrac{((n+1)p-2n)\alpha}{(n+1)p-n}.
\end{align*}

According to the assumption \eqref{eq:condition-f1}, we may assume for some positive constants $u_0>1$ and $C_0$
\begin{align*}
\abs{f(u)}\geq C_0^{-1}u^{p_0-1},\quad\forall u>u_0.
\end{align*}
Hence, we conclude that, if $f(u_0)>0$, there hold true
\begin{align*}
C_0^{-1}u^{p_0-1}\leq f(u)\leq \dfrac{f(u_0)}{u_0^{\alpha}} u^{\alpha},\quad\forall u\geq u_0
\end{align*}
and
\begin{align*}
f(u)\geq \dfrac{f(u_0)}{u_0^{\alpha}} u^{\alpha},\quad\forall u< u_0,
\end{align*}
since $f$ is subcritical. In particular, in this case, we must have $\alpha\geq p_0-1$.

Next, we need to consider the following three cases:
\medskip

{\bf Case 1. $t=2(p-1)$.}

In this case, it follows from \eqref{eq:crucial-0} that $u$ is a constant function under the assumption $p>n/2$ but without any growth estimate on $f$.
\medskip

{\bf Case 2. $t>2(p-1)$.}

In this case, we assume $f$ is nonnegative. According to \autoref{thm:main2}, we may assume $1<p<n$ and
\begin{align*}
0<\alpha<\dfrac{(n+1)p-n}{n-p}.
\end{align*}
Since
\begin{align*}
t=\dfrac{((n+1)p-2n)\alpha}{(n+1)p-n}>2(p-1),
\end{align*}
in view of the upper bound of $\alpha$ we get
\begin{align*}
2(p-1)<\dfrac{(n+1)p-2n}{n-p}
\end{align*}
which implies
$$p\geq\max\set{\frac{n+1}{2},\frac{2n}{n+1}}.$$
According to \autoref{thm:main2}, we know that $u$ is constant.
\medskip

{\bf Case 3. $t<2(p-1)$.}

In this case, we have
\begin{align*}
0<\dfrac{2(p-1)-t}{2(p_0-1)-t}<1,\quad\quad 0<\dfrac{2(p-1)-t}{2\alpha-t}<1,
\end{align*}
and
\begin{align*}
\dfrac{(2(p_0-1)-t)p}{p_0-p}-\dfrac{(2\alpha-t)p}{\alpha+1-p}=\dfrac{(2(p-1)-t)(\alpha+1-p_0)p}{(p_0-p)(\alpha+1-p)}\geq0,
\end{align*}
provided $p<p_0\leq\alpha+1$.

Now, there are two subcases which we need to deal with.
\medskip

{\bf Subcase 1. $f(u_0)\leq0$.}

In this case, we assume $p>n/2$. We may assume $f(u_0)<0$ and $p_0=\alpha$. Using H\"older's inequalities, we infer that for $$\gamma\geq\frac{(2(p_0-1)-t)p}{p_0-p}$$
there holds
\begin{align*}
\int_Mu^{2(p-1)-t}\abs{\nabla\eta}^{2p}\eta^{\gamma-2p}=&\int_{\set{u<u_0}}u^{2(p-1)-t}\abs{\nabla\eta}^{2p}\eta^{\gamma-2p}+\int_{\set{u\geq u_0}}u^{2(p-1)-t}\abs{\nabla\eta}^{2p}\eta^{\gamma-2p}\\
\leq&u_0^{2(p-1)-t}\int_{\set{u<u_0}}\abs{\nabla\eta}^{2p}\eta^{\gamma-2p}\\
&+\left(\int_{\set{u\geq u_0}}u^{2(p_0-1)-t}\eta^{\gamma}\right)^{\frac{2(p-1)-t}{2(p_0-1)-t}}\left(\int_{\set{u\geq u_0}}\abs{\nabla\eta}^{\frac{(2(p_0-1)-t)p}{p_0-p}}\eta^{\gamma-\frac{(2(p_0-1)-t)p}{p_0-p}}\right)^{\frac{2(p_0-p)}{2(p_0-1)-t}}\\
\leq&C\int_{\set{u<u_0}}\abs{\nabla\eta}^{2p}\eta^{\gamma-2p}\\
&+C\left(\int_{\set{u\geq u_0}}f(u)^2u^{-t}\eta^{\gamma}\right)^{\frac{2(p-1)-t}{2(p_0-1)-t}}\left(\int_{\set{u\geq u_0}}\abs{\nabla\eta}^{\frac{(2(p_0-1)-t)p}{p_0-p}}\eta^{\gamma-\frac{(2(p_0-1)-t)p}{p_0-p}}\right)^{\frac{2(p_0-p)}{2(p_0-1)-t}}.
\end{align*}
Thus, for
$$\gamma\geq\frac{(2(p_0-1)-t)p}{p_0-p},$$
we obtain from \eqref{add} and the above inequality
\begin{align*}
\int_{M}f(u)^2u^{-t}\eta^{\gamma}
\leq& C\gamma^{2p}\int_Mu^{2(p-1)-t}\abs{\nabla\eta}^{2p}\eta^{\gamma-2p}\\
\leq& C\gamma^{2p}\int_{\set{u<u_0}}\abs{\nabla\eta}^{2p}\eta^{\gamma-2p}\\
&+C\gamma^{2p}\left(\int_{\set{u\geq u_0}}f(u)^2u^{-t}\eta^{\gamma}\right)^{\frac{2(p-1)-t}{2(p_0-1)-t}}\left(\int_{\set{u\geq u_0}}\abs{\nabla\eta}^{\frac{(2(p_0-1)-t)p}{p_0-p}}\eta^{\gamma-\frac{(2(p_0-1)-t)p}{p_0-p}}\right)^{\frac{2(p_0-p)}{2(p_0-1)-t}}.
\end{align*}
By using Young's inequalities and the Bishop-Gromov volume comparison theorem, we infer from the above inequality that there exists some positive constant $C$ such that
\begin{align*}
\int_{B_{R/2}}f(u)^2u^{-t} \leq& C\left(R^{n-2p}+R^{n-\frac{(2(p_0-1)-t)p}{p_0-p}}\right).
\end{align*}
Noticing
\begin{align*}
n-\dfrac{(2(p_0-1)-t)p}{p_0-p}\leq& n-\dfrac{(2\alpha-t)p}{\alpha+1-p}\\
=&n-\dfrac{(n+1)p^2}{(n+1)p-n}\cdot\dfrac{1}{1-\frac{p-1}{\alpha}}\\
<&n-\dfrac{(n+1)p^2}{(n+1)p-n}\cdot\dfrac{1}{1-\frac{(p-1)(n-p)}{(n+1)p-n}}\\
=&-1,
\end{align*}
we conclude that
\begin{align*}
\int_{B_{R/2}}f(u)^2u^{-t} \leq C\left(R^{n-2p}+R^{-1}\right).
\end{align*}
Since $n<2p$, we conclude that $f(u)\equiv0$ on $M$ and $u$ is a $p$-harmonic function and $u$ is a constant function.
\medskip

{\bf Subcase 2. $f(u_0)>0$. }

We conclude that $f$ is positive everywhere  and $\alpha\geq p_0-1>p-1$. We estimate the following term by using H\"older's inequality
\begin{align*}
\int_{\set{u<u_0}}u^{2(p-1)-t}\abs{\nabla\eta}^{2p}\eta^{\gamma-2p}
\leq&\left(\int_{\set{u<u_0}}u^{2\alpha-t}\eta^{\gamma}\right)^{\frac{2(p-1)-t}{2\alpha-t}}
\left(\int_{\set{u<u_0}}\abs{\nabla\eta}^{\frac{(2\alpha-t)p}{\alpha+1-p}}
\eta^{\gamma-\frac{(2\alpha-t)p}{\alpha+1-p}}\right)^{\frac{2(\alpha+1-p)}{2\alpha-t}}\\
\leq&C\left(\int_{\set{u<u_0}}f(u)^2u^{-t}\eta^{\gamma}\right)^{\frac{2(p-1)-t}{2\alpha-t}}\left(\int_{\set{u<u_0}}
\abs{\nabla\eta}^{\frac{(2\alpha-t)p}{\alpha+1-p}}\eta^{\gamma-\frac{(2\alpha-t)p}{\alpha+1-p}}\right)^{\frac{2(\alpha+1-p)}{2\alpha-t}}.
\end{align*}
Thus, for
$$\gamma\geq\max\set{\frac{(2\alpha-t)p}{\alpha+1-p}, \, \frac{(2(p_0-1)-t)p}{p_0-p}}=\frac{(2(p_0-1)-t)p}{p_0-p},$$
we have
\begin{align*}
\int_{M}f(u)^2u^{-t}\eta^{\gamma}\leq& C\gamma^{2p}\int_Mu^{2(p-1)-t}\abs{\nabla\eta}^{2p}\eta^{\gamma-2p}\\
\leq&C\gamma^{2p}\int_Mu^{2(p-1)-t}\abs{\nabla\eta}^{2p}\eta^{\gamma-2p}\\
&+C\gamma^{2p}\left(\int_{\set{u\geq u_0}}f(u)^2u^{-t}\eta^{\gamma}\right)^{\frac{2(p-1)-t}{2(p_0-1)-t}}\left(\int_{\set{u\geq u_0}}\abs{\nabla\eta}^{\frac{(2(p_0-1)-t)p}{p_0-p}}\eta^{\gamma-\frac{(2(p_0-1)-t)p}{p_0-p}}\right)^{\frac{2(p_0-p)}{2(p_0-1)-t}}\\
\leq&C\gamma^{2p}\left(\int_{\set{u<u_0}}f(u)^2u^{-t}\eta^{\gamma}\right)^{\frac{2(p-1)-t}{2\alpha-t}}\left(\int_{\set{u<u_0}}
\abs{\nabla\eta}^{\frac{(2\alpha-t)p}{\alpha+1-p}}\eta^{\gamma-\frac{(2\alpha-t)p}{\alpha+1-p}}\right)^{\frac{2(\alpha+1-p)}{2\alpha-t}}\\
&+C\gamma^{2p}\left(\int_{\set{u\geq u_0}}f(u)^2u^{-t}\eta^{\gamma}\right)^{\frac{2(p-1)-t}{2(p_0-1)-t}}\left(\int_{\set{u\geq u_0}}\abs{\nabla\eta}^{\frac{(2(p_0-1)-t)p}{p_0-p}}\eta^{\gamma-\frac{(2(p_0-1)-t)p}{p_0-p}}\right)^{\frac{2(p_0-p)}{2(p_0-1)-t}}.
\end{align*}
Using Young's inequalities and the Bishop-Gromov volume comparison theorem, we deduce that there exists some positive constant $C$ such that
\begin{align*}
\int_{B_{R/2}}f(u)^2u^{-t} \leq& C\left(R^{n-\frac{(2\alpha-t)p}{\alpha+1-p}}+R^{n-\frac{(2(p_0-1)-t)p}{p_0-p}}\right)\leq CR^{n-\frac{(2\alpha-t)p}{\alpha+1-p}}\leq CR^{-1}.
\end{align*}
Hence, we conclude that $f(u)\equiv0$ which is impossible.
\end{proof}

\appendix
\section{A weak Harnack inequality for \texorpdfstring{$p$}{p}-superharmonic functions}
In this appendix, we will give a proof of the weak Harnack inequality  for positive $p$-superharmonic functions on a Riemannian manifold with nonnegative Ricci curvature. First, we  collect several  well-known facts:
\begin{enumerate}[\text{Fact} 1.]
\item Bishop-Gromov volume comparison theorem:
\begin{align*}
\dfrac{\mathrm{Vol}\left(B_{R}\right)}{\mathrm{Vol}\left(B_{r}\right)}\leq \dfrac{R^n}{r^n},\quad\forall 0<r<R.
\end{align*}
\item  Sobolev inequalities: for any $p\in[1,\infty)$ and any $\phi\in C_0^{\infty}\left(B_R\right)$
\begin{align*}
\left(\fint_{B_{R}}\abs{\psi}^{p\chi_{n,p}}\right)^{\frac{1}{p\chi_{n,p}}}\leq&C_{\mathcal{SD},p}R\left(\fint_{B_R}
\abs{\nabla\psi}^p\right)^{\frac{1}{p}}.
\end{align*}
Here $\chi_{n,p}=\frac{n}{n-p}$ if $1\leq p<n$ while $\chi_{n,p}$ can be any positive number which is larger than $1$.
\item Neumann  Poincar\'e inequalities: for any $p\in[1,\infty)$ and any $\phi\in C^{\infty}\left(B_R\right)$
\begin{align*}
\left(\fint_{B_R}\abs{\phi-\fint_{B_{R}}\phi}^p\right)^{1/p}\leq C_{\mathcal{PN},p}R\left(\fint_{B_R}\abs{\nabla\phi}^p\right)^{1/p}.
\end{align*}
\end{enumerate}

We need the following local maximum principle for positive and weak $p$-superharmonic functions.

\begin{lem}\label{lem:lmp}Let $M^n$ be a complete Riemannian manifold with dimension $n$ and nonnegative Ricci curvature. If $p>1$ and $u$ is a positive and weak p-superharmonic function, then
\begin{align}\label{eq:lmp}
\norm{u^{-1}}_{L^{\infty}\left(B_{ R/2}\right)}\leq&C_{\mathcal{LMP}}^{1/q}V_R^{-1/q}\norm{u^{-1}}_{L^{q}\left(B_{R}\right)},\quad\forall q>0,\quad\forall R>0.
\end{align}
Here $C_{\mathcal{LMP}}=C_{n,p,\chi_{n,p}}$ depends only on $n,p$ and $\chi_{n,p}$.
\end{lem}

\begin{proof}
For each real number $\beta>-\frac{p-1}{p}$, we consider the test function given by
$$\eta^{p}u^{1-(\beta+1)p}$$
where $\eta$ is a nonnegative cutoff function, then we get
\begin{align*}
0\leq&\int_M\left(-\Delta_pu\right)\eta^{p}u^{1-(\beta+1)p}\\
=&\left(1-(\beta+1)p\right)\int_M\abs{\nabla u}^pu^{-(\beta+1)p}\eta^{p}+p\int_M\abs{\nabla u}^{p-2}u^{1-(\beta+1)p}\eta^{p-1}\hin{\nabla u} {\nabla\eta}\\
\leq&\left(1-(\beta+1)p\right)\int_M\abs{\nabla u}^pu^{-(\beta+1)p}\eta^{p}+p\int_M\abs{\nabla u}^{p-1}u^{1-(\beta+1)p}\eta^{p-1}\abs{\nabla\eta}\\
\leq&\left(1-(\beta+1)p\right)\int_M\abs{\nabla u}^pu^{-(\beta+1)p}\eta^{p}+p\left(\int_M\abs{\nabla u}^pu^{-(\beta+1)p} \eta^{p}\right)^{\frac{p-1}{p}}\left(\int_Mu^{-\beta p}\abs{\nabla\eta}^p\right)^{\frac{1}{p}},
\end{align*}
which implies
\begin{align}\label{eq:weak-hanack-0}
\int_M\abs{\nabla u}^pu^{-(\beta+1)p}\eta^{p}\leq\left(\dfrac{p}{(\beta+1)p-1}\right)^p\int_Mu^{-\beta p}\abs{\nabla\eta}^p.
\end{align}
If $\beta\in\left(-\frac{p-1}{p},\, 0\right)\cup\left(0,\infty\right)$, then
\begin{align*}
\int_M\abs{\nabla\left(\eta u^{-\beta}\right)}^p\leq&2^{p-1}\left(\int_M\abs{\nabla u^{-\beta}}^p\eta^p+\int_Mu^{-\beta p}\abs{\nabla\eta}^p\right)\\
\leq&2^{p-1}\left(\abs{\beta}^p\left(\dfrac{p}{(\beta+1)p-1}\right)^p\int_Mu^{-\beta p}\abs{\nabla\eta}^p+\int_Mu^{-\beta p} \abs{\nabla\eta}^p\right).
\end{align*}
Thus
\begin{align*}
\int_M\abs{\nabla\left(\eta u^{-\beta}\right)}^p\leq&2^{p-1}\left(\left(\dfrac{\abs{\beta}p}{(\beta+1)p-1}\right)^p+1\right)\int_Mu^{-\beta p} \abs{\nabla\eta}^p.
\end{align*}
In particular, for every $\beta>0$
\begin{align}\label{eq:weak-harnack-1}
\int_M\abs{\nabla\left(\eta u^{-\beta}\right)}^p\leq 2^p\int_Mu^{-\beta p}\abs{\nabla\eta}^p.
\end{align}

We can apply the Moser iteration to obtain the local maximum principle. To see this, we chose sequences of $R_i$ and $\beta_i$ such that
\begin{align*}
R_i=\left(1+2^{-i}\right)R/2,\quad \beta_i=\beta_0\chi_{n,p}^i,\quad \beta_0>0,\quad i=0,1,2,\dotsc.
\end{align*}
For each $i$, we can choose a cutoff function  $\eta_i\in C_0^{\infty}\left(B_{R_i}\right)$ satisfying
\begin{align*}
0\leq\eta_i\leq 1,\quad \eta_i\vert_{B_{R_{i+1}}}=1,\quad\abs{\nabla\eta_i}\leq 2^{3+i}R^{-1},\quad i=0,1,2,\dotsc.
\end{align*}
Consequently, it follows from \eqref{eq:weak-harnack-1} and the Sobolev inequality \eqref{eq:sobolev-p}
\begin{align*}
\norm{u^{-\beta_i}}_{L^{p\chi}\left(B_{R_{i+1}}\right)}\leq& \norm{\eta_i u^{-\beta_i}}_{L^{p\chi}\left(B_{R_i}\right)}\\
\leq& C_{\mathcal{SD},p}V_{B_{R_i}}^{-\frac{\chi_{n,p}-1}{p\chi_{n,p}}}R\norm{\nabla\left(\eta_i u^{-\beta_i}\right)}_{L^{p}\left(B_{R_i}\right)}\\
\leq& 2C_{\mathcal{SD},p}V_{B_{R_i}}^{-\frac{\chi_{n,p}-1}{p\chi_{n,p}}}R\norm{u^{-\beta_i}\abs{\nabla\eta_i}}_{L^{p}\left(B_{R_i}\right)}\\
\leq&2^{4+i}C_{\mathcal{SD},p}V_{B_{R_i}}^{-\frac{\chi_{n,p}-1}{p\chi_{n,p}}}\norm{u^{-\beta_i}}_{L^{p}\left(B_{R_i}\right)}
\end{align*}
which implies
\begin{align*}
\norm{u^{-p}}_{L^{\beta_{i+1}}\left(B_{R_{i+1}}\right)}\leq \left(2^{4+i}C_{\mathcal{SD},p}V_{B_{R_i}}^{-\frac{\chi_{n,p}-1}{p\chi_{n,p}}}\right)^{\frac{p}{\beta_0\chi^i}}
\norm{u^{-p}}_{L^{\beta_i}\left(B_{R_i}\right)}.
\end{align*}
Applying the Bishop-Gromov volume comparison theorem, we get
\begin{align*}
\norm{u^{-p}}_{L^{\infty}\left(B_{R/2}\right)}\leq& \Pi_{i=0}^{\infty}\left(2^{4+i}C_{\mathcal{SD},p}V_{B_{R_i}}^{-\frac{\chi_{n,p}-1}{p\chi_{n,p}}}\right)^{\frac{p}{\beta_0\chi^i}}\norm{u^{-p}}_{L^{\beta_0}\left(B_{R}\right)}\\
\leq& \Pi_{i=0}^{\infty}\left(2^{4+i}2^{\frac{n\left(\chi_{n,p}-1\right)}{p\chi_{n,p}}}
C_{\mathcal{SD},p}V_{B_{R}}^{-\frac{\chi_{n,p}-1}{p\chi_{n,p}}}\right)^{\frac{p}{\beta_0\chi^i}}\norm{u^{-p}}_{L^{\beta_0}\left(B_{R}\right)}\\
\leq&C_{n,p,\chi_{n,p}}^{\frac{1}{\beta_0}}C_{\mathcal{SD},p}^{\frac{p\chi_{n,p}}{\beta_0(\chi_{n,p}-1)}}V_{R}^{-\frac{1}{\beta_0}}
\norm{u^{-p}}_{L^{\beta_0}\left(B_{R}\right)}.
\end{align*}
We obtain the local maximum principle \eqref{eq:lmp}.
\end{proof}

Now we can prove the weak Harnack inequality for $p$-superharmonic functions.

\begin{proof}[Proof of \autoref{lem:weak-hanack}]

The main step is to prove the following: for some positive constant $\beta_0$
\begin{align}\label{eq:weakhanack1}
    \int_{B_{R/2}}u^{-\beta_0}\int_{B_{R/2}}u^{\beta_0}\leq CR^{2n}.
\end{align}

It follows from \eqref{eq:weak-hanack-0}
\begin{align}\label{eq:weak-hanack-w0}
\int_M\abs{\nabla \ln u}^p\eta^p\leq\left(\dfrac{p}{p-1}\right)^{p}\int_M\abs{\nabla\eta}^p.
\end{align}
Consequently, the Bisho-Gromov comparison theorem yields
\begin{align*}
\int_{B_{ R}}\abs{\nabla \ln u}^p\leq C_{n,p}R^{-p}V_{R}.
\end{align*}
Denote $w=\ln u-\fint_{B_{R}}\ln u$. We obtain from the Neumann Poincar\'e inequality
\begin{align*}
\int_{B_{ R}}\abs{w}^p\leq&C_{\mathcal{PN},p}^pR^p\int_{B_{ R}}\abs{\nabla w}^p\leq C_{n,p}V_R.
\end{align*}

For each positive number $\gamma\geq1$, replace $\eta$ by $\eta \abs{w}^{\gamma}$ in \eqref{eq:weak-hanack-w0} to obtain
\begin{align*}
\int_M\abs{\nabla w}^p\abs{w}^{p\gamma}\eta^p\leq&\left(\dfrac{p}{p-1}\right)^{p}\int_M\abs{\nabla\left(\eta\abs{w}^{\gamma}\right)}^p\\
\leq&\left(\dfrac{p}{p-1}\right)^{p}2^{p-1}\left(\gamma^p\int_M\abs{w}^{p(\gamma-1)}\abs{\nabla w}^p\eta^p+\int_M\abs{w}^{p\gamma}\abs{\nabla\eta}^p\right)\\
\leq&\left(\dfrac{p}{p-1}\right)^{p}2^{p-1}\left(\gamma^p\left(\int_M\abs{w}^{p\gamma}\abs{\nabla w}^p\eta^p\right)^{\frac{\gamma-1}{\gamma}}\left(\int_M\abs{\nabla w}^p\eta^p\right)^{\frac{1}{\gamma}}+\int_M\abs{w}^{p\gamma}\abs{\nabla\eta}^p\right)\\
\leq&\dfrac{1}{2}\int_M\abs{\nabla w}^p\abs{w}^{p(\gamma-1)}\eta^p+\dfrac{\left(C_p\gamma^p\right)^{\gamma}}{\gamma}\int_M\abs{\nabla w}^p\eta^p+C_p\int_M\abs{w}^{p\gamma}\abs{\nabla\eta}^p.
\end{align*}
Thus
\begin{align*}
\int_M\abs{\nabla w}^p\abs{w}^{p\gamma}\eta^p\leq&\dfrac{\left(C_p\gamma^p\right)^{\gamma}}{\gamma}\int_M\abs{\nabla w}^p\eta^p+C_p\int_M\abs{w}^{p\gamma}\abs{\nabla\eta}^p.
\end{align*}
We obtain
\begin{align*}
\int_M\abs{\nabla\left(\eta\abs{w}^{\gamma+1}\right)}^p\leq&2^{p-1}\left(\int_M\abs{w}^{(\gamma+1)p}\abs{\nabla\eta}^p+(\gamma+1)^p\int_M\abs{w}^{p\gamma}\abs{\nabla w}^{p}\eta^p\right)\\
\leq&C_p^{\gamma}\gamma^{(\gamma+1)p-1}\int_M\abs{\nabla w}^p\eta^p+C_p\gamma^p\int_M\abs{w}^{p\gamma}\abs{\nabla\eta}^p+C_p\int_M\abs{w}^{(\gamma+1)p}\abs{\nabla\eta}^p\\
\leq& C_p^{\gamma}\gamma^{(\gamma+1)p-1}\int_M\abs{\nabla w}^p\eta^p +C_p\gamma^p\left(\int_M\abs{\nabla\eta}^p\right)^{\frac{1}{\gamma+1}}\left(\int_M\abs{w}^{(\gamma+1)p}
\abs{\nabla\eta}^p\right)^{\frac{\gamma}{\gamma+1}}\\
&+C_p\int_M\abs{w}^{(\gamma+1)p}\abs{\nabla\eta}^p\\
\leq& C_p^{\gamma}\gamma^{(\gamma+1)p-1}\int_M\abs{\nabla\eta}^p+C_p\int_M\abs{w}^{(\gamma+1)p}\abs{\nabla\eta}^p.
\end{align*}
Thus, for $\gamma\geq2$
\begin{align*}
\int_M\abs{\nabla\left(\eta\abs{w}^{\gamma}\right)}^p\leq C_p^{\gamma}\gamma^{\gamma p}\int_M\abs{\nabla\eta}^p+C_p\int_M\abs{w}^{\gamma p}\abs{\nabla\eta}^p.
\end{align*}

Set
\begin{align*}
R_i=\left(1+2^{-i}\right)R/2,\quad \gamma_i=2\chi_{n,p}^i,\quad i=0,1,2,\dots.
\end{align*}
For each $i$, choose $\eta_i\in C_0^{\infty}\left(B_{R_i}\right)$ satisfying
\begin{align*}
0\leq\eta_i\leq 1,\quad \eta_i\vert_{B_{R_{i+1}}}=1,\quad\abs{\nabla\eta_i}\leq 2^{3+i}R^{-1},\quad i=0,1,2,\dots.
\end{align*}
Applying the Sobolev inequality \eqref{eq:sobolev-p}, we obtain
\begin{align*}
\left(\int_{B_{R_i}}\left(\eta_i\abs{w}^{\gamma_i}\right)^{p\chi}\right)^{\frac{1}{\chi}}\leq&C_{n,p,\chi_{n,p}}
V_{R_i}^{-\frac{\chi_{n,p}-1}{\chi_{n,p}}}R^p\int_{B_{R_i}}\abs{\nabla\left(\eta_i\abs{w}^{\gamma_i}\right)}^p\\
\leq&C_{n,p,\chi_{n,p}}V_{R_i}^{-\frac{\chi_{n,p}-1}{\chi_{n,p}}}R^p\left(C_p^{\gamma_i}\gamma_i^{\gamma_i p}\int_{B_{R_i}}\abs{\nabla\eta_i}^p+C_p\int_{B_{R_i}}\abs{w}^{\gamma_i p}\abs{\nabla\eta_i}^p\right).
\end{align*}
We get
\begin{align*}
\left(\int_{B_{R_{i+1}}}\abs{w}^{p\gamma_{i+1}}\right)^{\frac{1}{\chi}}\leq&C_{n,p,\chi_{n,p}}V_{R_i}^{-\frac{\chi_{n,p}-1}{\chi_{n,p}}}2^{ip}
\left(\int_{B_{R_i}}\abs{w}^{\gamma_i p}+C_p^{\gamma_i}\gamma_i^{\gamma_i p}V_{R_i}\right).
\end{align*}
Hence,
\begin{align*}
\norm{w}_{L^{p\gamma_{i+1}}\left(B_{R_{i+1}}\right)}\leq &\left(C_{n,p,\chi_{n,p}}V_{R_i}^{-\frac{\chi_{n,p}-1}{p\chi_{n,p}}}2^i\right)^{\frac{1}{\gamma_i}}
\left(\norm{w}_{L^{p\gamma_{i}}\left(B_{R_{i}}\right)}+C_p\gamma_iV_{R_i}^{\frac{1}{p\gamma_i}}\right),\quad i=0,1,2,\dots.
\end{align*}
Denote
\begin{align*}
x_i=V_{R_i}^{-\frac{1}{p\gamma_i}}\norm{w}_{L^{p\gamma_{i}}\left(B_{R_{i}}\right)},\quad a_i=\left(C_{n,p,\chi}2^{i}\right)^{\frac{1}{\gamma_i}},\quad b_i=C_p\gamma_i.
\end{align*}
Then the Bishop-Gromov comparison theorem yields
\begin{align*}
x_{i+1}\leq a_i\left(x_i+b_i\right),\quad i=0,1,2,\dots.
\end{align*}
Iterating the above inequality and observing
\begin{align*}
\Pi_{i=0}^{\infty}a_i= &\left(C_{n,p,\chi_{n,p}}\right)^{\sum_{i=0}^{\infty}\frac{1}{\chi_{n,p}^i}}2^{\sum_{i=0}^{\infty}\frac{i}{\chi_{n,p}^i}}
=\left(C_{n,p,\chi_{n,p}}\right)^{\frac{\chi_{n,p}}{\chi_{n,p}-1}}2^{\frac{\chi_{n,p}}{(\chi_{n,p}-1)^2}},
\end{align*}
we obtain
\begin{align*}
x_{i+1}\leq& \Pi_{j=0}^ia_jx_0+\sum_{k=0}^i\Pi_{j=k}^ia_jb_{k}\leq C_{n,p,\chi_{n,p}}\left(x_0+\chi_{n,p}^{i+1}\right).
\end{align*}

Now for $\gamma\geq2$ there exists a $i$ such that $p\chi_{n,p}^i\leq \beta<p\chi_{n,p}^{i+1}$.
The Bishop-Gromov volume comparison theorem yields
\begin{align*}
\left(\fint_{B_{R/2}}\abs{w}^{\beta}\right)^{\frac{1}{\beta}}\leq&\left(\dfrac{R_{i+1}}{R/2} \right)^{\frac{n}{\gamma_{i+1}}}\left(\fint_{B_{R_{i+1}}}\abs{w}^{\gamma_{i+1}}\right)^{\frac{1}{\gamma_{i+1}}}\\
\leq& C_{n,p,\chi_{n,p}}\left(\fint_{B_{R}}\abs{\nabla w}^{2p}+\chi_{n,p}^{i+1}\right)\\
\leq& C_{n,p,\chi_{n,p}}\beta.
\end{align*}
Thus
\begin{align*}
\fint_{B_{R/2}}\abs{w}^{\beta}\leq C_{n,p,\chi_{n,p}}^{\beta}\beta^{\beta}\leq C_{n,p,\chi_{n,p}}^{\beta}e^{\beta}\beta!.
\end{align*}
We conclude that for some positive number $\beta_0$
\begin{align*}
    \fint_{B_{R/2}}e^{\beta_0\abs{w}}\leq 2.
\end{align*}
Consequently, we get \eqref{eq:weakhanack1}.

Now we can prove the weak Harnack inequality \eqref{eq:weak-hanack} as follows. Firstly, according the local maximum principle \autoref{lem:lmp} and \eqref{eq:weakhanack1}, we conclude that
\begin{align*}
\inf_{B_{R/4}}u=&\left(\sup_{B_{R/4}}u^{-1}\right)^{-1}\geq\left(CR^{-\frac{n}{\beta_0}}\norm{u^{-1}}_{L^{\beta_0}\left(B_{R/2}\right)}\right)^{-1}\\
\geq&C^{-1}R^{-\frac{n}{\beta_0}}\norm{u}_{L^{\beta_0}\left(B_{R/2}\right)}.
\end{align*}
That is,  the weak Harnack inequality  holds for $q=\beta_0$.

Next we will prove that the weak Harnack inequality \eqref{eq:weak-hanack}  holds for every $0<q<(p-1)\chi_{n,p}$. It follows from \eqref{eq:weak-hanack-0} that for each $0<\epsilon<\frac{p-1}{p}$,

\begin{align}\label{eq:weak-harnack-3}
\int_M\abs{\nabla\left(\eta u^{\epsilon}\right)}^p\leq \dfrac{C_p}{(1-\epsilon)p-1}\int_Mu^{\epsilon p}\abs{\nabla\eta}^p.
\end{align}
It follows from \eqref{eq:weak-harnack-3} and the Sobolev inequality \eqref{eq:sobolev-p}, we obtain
\begin{align}\label{eq:weak-hanack-4}
\left(\int_{B_{R/2}}u^{p\epsilon\chi_{n,p}}\right)^{\frac{1}{\chi_{n,p}}}\leq& \dfrac{C_{n,p,\chi}}{(1-\epsilon)p-1} V_{R}^{-\frac{\chi_{n,p}-1}{\chi_{n,p}}}\int_{B_{R}}u^{p\epsilon}.
\end{align}
We may assume
\begin{align*}
0<\beta_0<(p-1)\chi_{n,p}.
\end{align*}
First we assume
\begin{align*}
\beta_0<p-1.
\end{align*}
Choose $\epsilon\in\left(0,\,\frac{p-1}{p}\right)$ such that
\begin{align*}
p\epsilon=\beta_0.
\end{align*}
Then
\begin{align*}
\left(\int_{B_{R/2}}u^{p\beta_0}\right)^{\frac{1}{\chi_{n,p}}}\leq& \dfrac{C_{n,p,\chi_{n,p}}}{(1-\epsilon)p-1} V_{R}^{-\frac{\chi_{n,p}-1}{\chi_{n,p}}}\int_{B_{R}}u^{\beta_0}.
\end{align*}
Thus, we may assume the weak Hanack equality holds for $q=\beta_1\coloneqq \beta_0\chi_{n,p}$. After finitely many iterations, we may assume $\beta_0\geq p-1$. It then follows from \eqref{eq:weak-hanack-4} that the weak Harnack inequality \eqref{eq:weak-hanack}  holds for each $q\in\left(0,(p-1)\chi_{n,p}\right)$.
\end{proof}
\medskip

\noindent {\it\bf{Acknowledgements}}: Y. Wang is supported by National Natural Science Foundation of China (Grant No. 12431003).
\medskip

\noindent {\it\bf{Conflict of interest statement}}: The authors declare that there are no conflict of interests.

\medskip
\noindent {\it\bf{Data availability statement}}: No data was used in the research described in this paper.

\medskip



\end{document}